\documentclass[a4paper,10pt]{article}

\def\zibreport{1}
\def\longtitle{Linear Programming using Limited-Precision Oracles}
\def\shortfunding{Ambros Gleixner was supported by the Research Campus MODAL
  \textit{Mathematical Optimization and Data Analysis Laboratories} funded
  by the German Ministry of Education and Research (BMBF grant number 05M14ZAM).}
\def\shortauthors{Ambros Gleixner, Daniel E. Steffy}

\usepackage{geometry}
\usepackage[colorlinks,citecolor=blue,linkcolor=blue,urlcolor=blue,
  pdftitle={\longtitle},
  pdfauthor={\shortauthors},
  pdfsubject={\longtitle}]{hyperref}

\usepackage{amsmath,amsthm}
\usepackage[numbers]{natbib}

\usepackage[ruled,linesnumbered,vlined]{algorithm2e}
\SetAlgoCaptionSeparator{:}

\SetKwInOut{Input}{input}
\SetKwInOut{Output}{output}
\SetKwInOut{Params}{parameters}
\SetKwInOut{Bla}{paramst}
\SetKw{myElse}{else}

\DontPrintSemicolon
 \newcommand{\name}[1]{\mbox{#1}\xspace}

\newcommand{\netlib}{\name{Netlib}}

\newcommand{\qsoptex}{\name{QSopt\_ex}}

\newcommand{\soplex}{\name{SoPlex}}

\newcommand{\soplexrec}{\name{SoPlex$_{\text{rec}}$}}
\newcommand{\soplexfac}{\name{SoPlex$_{\text{fac}}$}}

\newcommand{\gcc}{\name{GCC}}
\newcommand{\eglib}{\name{EGlib}}
\newcommand{\gmp}{\name{GMP}}

\newcommand{\miplib}{\name{MIPLIB}}

\newcommand{\coral}{\name{COR@L}}

 \usepackage{amsmath,amsfonts,amssymb}
\usepackage{amssymb}
\usepackage{mathtools} \usepackage[nice]{nicefrac}

\newcommand{\T}{T}\renewcommand{\leq}{\leqslant}
\renewcommand{\geq}{\geqslant}
\renewcommand{\phi}{\varphi}
\renewcommand{\epsilon}{\varepsilon}
\renewcommand{\Leftrightarrow}{\xLeftrightarrow{\quad}}
\newcommand{\without}{\setminus}

\newcommand{\order}{\ensuremath{O}}
\newcommand{\size}[1]{\ensuremath{\langle#1\rangle}}
\newcommand{\ceil}[1]{\ensuremath{\lceil#1\rceil}}
\newcommand{\floor}[1]{\ensuremath{\lfloor#1\rfloor}}

\newcommand{\B}{\mathcal{B}}

\newcommand{\Q}{\mathbb{Q}}
\newcommand{\Z}{\mathbb{Z}}
\newcommand{\R}{\mathbb{R}}

\newcommand{\N}{\mathcal{N}}

\newcommand{\nonneg}[1]{\ensuremath{#1_{\ge 0}}\xspace}

\newcommand{\unitmatrix}{I}
\newcommand{\lb}{\ell}

\usepackage{bbm}
\newcommand{\one}{\mathbbm{1}}

\newcommand{\enifed}{\ensuremath{=:}}
\newcommand{\assign}{\mbox{$\,\leftarrow\,$}}
\newcommand{\norm}[1]{\lVert#1\rVert}

\newcommand{\maxnorm}[1]{\norm{#1}_\infty}

\newcommand{\abs}[1]{\lvert#1\rvert}
\newcommand{\bigabs}[1]{\big\lvert#1\big\rvert}
\newcommand{\round}{round\_to\_denom}

\newcommand{\fpdigits}{p}
\newcommand{\fptol}{\eta}
\newcommand{\termtol}{\tau}
\newcommand{\scalpow}{D}
\newcommand{\scalfac}{\Delta}
\newcommand{\fpnumbers}{\mathbb{F}}

\newtheoremstyle{itstyle}  {\bigskipamount}  {\bigskipamount}  {\normalfont\itshape}  {}  {\normalfont}  { }  { }  {\textbf{\thmname{#1}\thmnumber{ #2}}\thmnote{ (#3)}\textbf{.}}
\newtheoremstyle{upstyle}  {\bigskipamount}  {\bigskipamount}  {\normalfont}  {}  {\normalfont\itshape}  { }  { }  {\textbf{\thmname{#1}\thmnumber{ #2}}\thmnote{ (#3)}\textbf{.}}
\theoremstyle{itstyle}
\newtheorem{definition}{Definition}
\newtheorem{lemma}{Lemma}
\newtheorem{theorem}{Theorem}

\theoremstyle{upstyle}
\newtheorem{example}{Example}

\renewenvironment{proof}[1][Proof]{\noindent\textit{#1.} }{\bigskip}
\newenvironment{claim}[1]{\bigskip\noindent\textit{#1.}}{\medskip}

 \usepackage{booktabs}		\usepackage{longtable}
\usepackage{lscape}

\newcommand{\timecol}{\ensuremath{t}}

\newcommand{\numinstcol}{$\#$inst}

\newcommand{\preccol}{prec}
\newcommand{\irrefcol}{$\#$ref}

\newcommand{\itercol}{$\#$iter}

\newcommand{\speedupcol}{\ensuremath{\Delta t}}

\newcommand{\faccol}{$\#$fac}
\newcommand{\factimecol}{\ensuremath{t_{\text{fac}}}}
\newcommand{\reccol}{$\#$rec}
\newcommand{\rectimecol}{\ensuremath{t_{\text{rec}}}}
\newcommand{\dlcmcol}{dlcm}

\newcommand{\timebracketgeq}[1]{$[#1,7200]$}

 \usepackage{etex} 
\usepackage{tikz}
\usepackage{pgfplots}

\usetikzlibrary{arrows}

\colorlet{plot1}{cyan!60!black}
\colorlet{plot2}{green!80!black!80!red}
\colorlet{plot3}{brown}
\colorlet{plot4}{purple!70!blue}
\tikzstyle{plot1time} = [draw=plot1, thick]
\tikzstyle{plot1nodes} = [draw=plot1, thick, dashed]
\tikzstyle{plot2time} = [draw=plot2, thick]
\tikzstyle{plot2nodes} = [draw=plot2, thick, dashed]
\tikzstyle{plot3time} = [draw=plot3, thick]
\tikzstyle{plot3nodes} = [draw=plot3, thick, dotted]

\pgfplotsset{histostyle/.style={    width=0.8\textwidth, height=3cm,
    scale only axis,
        xmin=0, xmax=100,
    axis x line=bottom,
    enlarge x limits=0.005,
    every x tick/.append style={draw=black},
    xtick align = inside,
    xtick = {0,20,40,60,80,100},
    xticklabels = {0\%,20\%,40\%,60\%,80\%,100\%},
    scaled ticks = false,
    minor tick num=3,
    minor tick length=3,
    major tick length=3,
    every tick label/.append style={font=\scriptsize},
    every outer x axis line/.style = {},
        ymin=0,
    ylabel={\raisebox{-1.75em}[0em][0em]{\scriptsize instances}},
    every y tick label/.append style={color=white},
    every y tick/.append style={draw=none},
    every axis/.append style={draw=white},
    every node near coord/.append style={black,font=\scriptsize,rotate=90},
    nodes near coords align=right,
}}

\pgfplotsset{histoplot/.style={    ybar,
    bar width=10,
    nodes near coords,
    point meta=y,
    draw=blue!15!white,
    fill=blue!15!white
}}

\pgfplotsset{loghistostyle/.style={    width=0.8\textwidth, height=3cm,
    scale only axis,
        xmin=0, xmax=8,
    axis x line=bottom,
    enlarge x limits=0.005,
    every x tick/.append style={draw=black},
    xtick align = inside,
    xtick = {0,1,2,3,4,5,6,7,8},
    xticklabels = {\raisebox{0ex}[1.8ex]{0},\raisebox{0ex}[1.8ex]{1},\raisebox{0ex}[1.8ex]{10},$10^2$,$10^3$,$10^4$,$10^5$,$10^6$,$10^7$,$10^8$},
    scaled ticks = false,
    minor tick num=0,
    minor tick length=3,
    major tick length=2,
    every tick label/.append style={font=\tiny},
    every outer x axis line/.style = {},
        ymin=0,
    ylabel={\raisebox{-1.75em}[0em][0em]{\scriptsize instances}},
    every y tick label/.append style={color=white},
    every y tick/.append style={draw=none},
    every axis/.append style={draw=white},
    every node near coord/.append style={black,font=\scriptsize,rotate=90},
    nodes near coords align=right,
}}

\pgfplotsset{plotstyle/.style={const plot,
         width = 0.95\textwidth, height = 15em,
         scaled ticks = false,
                  major tick length = 5,
         minor tick length = 3,
         tick label style = {font=\scriptsize},
                  xmin = 0,          axis x line = bottom,
         every outer x axis line/.style = {},
         minor x tick num = 5,
         xtick align = outside,
         xtick = {0,1200,2400,3600,4800,6000,7200},
         enlarge x limits=0.005,
                  ymin = -0.903089987,
         axis y line = left,
         ytick = {-0.602059991,-0.301029996,0,0.301029996,0.602059991,0.903089987},          yticklabels = {0.25,0.5,1,2,4,8},          every outer y axis line/.style = {white},
         every y tick/.style = {white},
         extra y tick style = {grid=none, tick label style={font=\small, anchor=south east}},
         ymajorgrids,
         every axis grid/.append style = {color=darkgray,dotted},
                  no markers,
         area legend, legend columns = 3,
         every axis legend/.append style = { at={(0.5,0.92)}, anchor=south, font=\small, draw=none, rounded corners=2pt}
}}

\pgfplotsset{objdevstyle/.style={         width = \textwidth, height = 15em,
         scaled ticks = false,
                  major tick length = 5,
         minor tick length = 3,
         tick label style = {font=\scriptsize},
                  xmin = 1.03026, xmax = 1.03193,
         axis x line = bottom,
         every outer x axis line/.style = {},
         minor x tick num = 5,
         xtick align = outside,
         xtick = {1.03026,1.031095,1.03193},
         xticklabels = {1.03026,1.031095,1.03193},
         enlarge x limits=0.02,
                  ymin = -20, ymax = 2,
         axis y line = left,
         ytick = {-15,-10,-5,0},          yticklabels = {$10^{-15}$,$10^{-10}$,$10^{-5}$,1.0},          every outer y axis line/.style = {white},
         every y tick/.style = {white},
         extra y tick style = {grid=none, tick label style={font=\small, anchor=south east}},
         ymajorgrids,
         every axis grid/.append style = {color=darkgray,dotted},
                           area legend, legend columns = 3,
         every axis legend/.append style = { at={(0.5,0.92)}, anchor=south, font=\small, draw=none, rounded corners=2pt}
}}

\tikzstyle{fbacpx} = [draw=plot3, thick, mark=*, mark size=0.7, mark options={fill=plot3, solid}]
\tikzstyle{fbalong} = [draw=plot1, thick, mark=*, mark size=0.7, mark options={fill=plot1, solid}] \tikzstyle{fbair} = [draw=plot2, thick, mark=*, mark size=0.7, mark options={fill=plot2, solid}]

\usepackage{tkz-graph}
\colorlet{COLOR0}{red}
\colorlet{COLOR1}{black!0}
\colorlet{COLOR2}{black}
\tikzstyle{ucedge} += [line width=0.01, color=COLOR2]
\tikzstyle{ucnode} += [line width=0.01]
 \newcommand{\myorcidlink}[1]{\,\href{https://orcid.org/#1}{\raisebox{-0.45ex}{\includegraphics[width=1.8ex]{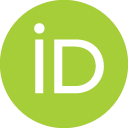}}}}
\newcommand{\doilink}[1]{\href{http://dx.doi.org/#1}{{\textcolor{blue}{DOI:\allowbreak#1}}}}
\newcommand{\myorcid}[1]{ORCID \href{https://orcid.org/#1}{#1}}

\ifthenelse{\zibreport = 1}{
  \let\pdfoutorg\pdfoutput
  \let\pdfoutput\undefined
  \usepackage{zibtitlepage}
  \let\pdfoutput\pdfoutorg

\ZTPAuthor{Ambros Gleixner\protect\myorcidlink{0000-0003-0391-5903} and Daniel E. Steffy\protect\myorcidlink{0000-0002-0370-9555}}
\ZTPTitle{\bf Linear Programming using Limited-Precision Oracles}
\ZTPInfo{The content of this report is also available as publication in the journal {\em Mathematical Programming}.
  Please always cite as:\\[1ex] A.~Gleixner, D.E.~Steffy. Linear Programming using Limited-Precision Oracles.
  {\em Mathematical Programming}, 2019,
  \doilink{10.1007/s10107-019-01444-6}}
\ZTPNumber{19-57}
\ZTPMonth{November}
\ZTPYear{2019}
}{}

\begin{document}

\title{\longtitle}

\author{Ambros Gleixner  \thanks{\itshape Zuse Institute Berlin, Department of Mathematical Optimization,
    Takustr.~7, 14195~Berlin, Germany,
    E-mail \nolinkurl{gleixner@zib.de}, \myorcid{0000-0003-0391-5903}\vspace*{0.5ex}}
  \and
  Daniel E.~Steffy  \thanks{\itshape Mathematics and Statistics, Oakland University, Rochester, Michigan 48309, USA,
    E-mail \nolinkurl{steffy@oakland.edu}, \myorcid{0000-0002-0370-9555}\vspace*{1.5ex}
    \footnoterule\upshape\noindent
    \shortfunding
  }}

\date{November 30, 2019}

\ifthenelse{\zibreport = 1}{\zibtitlepage}{}

\newgeometry{left=38mm,right=38mm,top=35mm}

\maketitle

\paragraph{\bf Abstract}

Since the elimination algorithm of Fourier and Motzkin, many different methods
have been developed for solving linear programs.
When analyzing the time complexity of LP algorithms, it is typically either 
assumed that calculations are performed exactly and bounds are derived on the 
number of elementary arithmetic operations necessary, or the cost of all 
arithmetic operations is considered through a bit-complexity analysis.
Yet in practice, implementations typically use limited-precision arithmetic.
In this paper we introduce the idea of a limited-precision LP oracle and study 
how such an oracle could be used within a larger framework to compute exact 
precision solutions to LPs.
Under mild assumptions, it is shown that a polynomial number of calls
to such an oracle and a polynomial number of bit operations, is 
sufficient to compute an exact solution to an LP.
This work provides a foundation for understanding and analyzing the behavior of 
the methods that are currently most effective in practice for solving LPs 
exactly.



\section{Introduction}
\label{sec:intro}

This paper studies algorithms for solving linear programming problems (LPs)
exactly over the rational numbers.
The focus lies on methods that employ a limited-precision LP
oracle---an oracle that is capable of providing approximate primal-dual solutions.
Connections will be made to previous theoretical and practical studies.
We consider linear programs of the following \emph{standard form}:
\begin{equation}\tag{$P$}\label{eq:standardlp}
 \min\{ c^\T x \mid Ax = b,~x \ge \lb \}
\end{equation}
where $A$ is an $m \times n$ matrix of full row rank with $m \leq n$, $x$ is a
vector of variables, $c$ is the objective vector and $\lb$ is a vector of lower 
bounds.
This form of LP is convenient for describing the algorithms, any of which can be 
adapted to handle alternative forms.  
It is assumed that all input data is rational and our goal is to compute exact
solutions over the rational numbers.
We note that although the assumption of rational data is common in the 
literature and holds for most applications, there are some applications where 
irrational data arises naturally (e.g. from geometric structures).
In such cases where irrational data is necessary, many of the algorithms 
described herein may still be of use, for example to compute highly accurate 
approximate solutions.
In the following, we survey relevant background material and 
previous work.

\subsection{Basic Terminology}

We assume that the reader is familiar with fundamental results
related to linear optimization, such as those presented in
\cite{GroetschelLovaszSchrijver1988,Schrijver1986}.
The following notation will be used throughout the paper.
For a matrix $A$, and subsets $J$, $K$ of the rows and columns,
respectively, we use $A_{JK}$ to denote the submatrix of $A$ formed by
taking entries whose row and column indices are in those sets.
When $J$ or $K$ consists of a single element $i$ we will use $i$
in place of $\{i\}$, and `$\cdot$' is used to represent all rows or columns.
The unit matrix of dimension~$n$ is denoted by~$I_n$.

It is well known that if an LP has a bounded objective value
then it has an optimal \emph{basic solution} $x^*$ of the following form.
For a subset $\B$ of $m$ linearly independent columns $A$, set
$x_\B^* := A_{\cdot\B}^{-1} (b - \sum_{i\not\in\B} A_{\cdot i} \lb_i)$
and $x_i^* := \lb_i$ for all $i\notin \B$.
Generally, such a set $\B$ is called a \emph{basis}
and the matrix $A_{\cdot\B}$ is the \emph{basis matrix} associated with $\B$.

We denote the \emph{maximum norm} of a vector~$x\in\R^n$ by
$
  \maxnorm{x} := \max_{i=1,\ldots,n} \abs{x_i}.
$
The corresponding \emph{row sum norm} of a matrix~$A\in\R^{m\times n}$ given by
\begin{equation}
  \maxnorm{A} := \max_{i=1,\ldots,m} \sum_{j=1}^{n}\abs{A_{ij}},
\end{equation}
is compatible with the maximum norm in the sense that
$
  \maxnorm{Ax} \leq \maxnorm{A} \maxnorm{x}
$
for all~$A\in\R^{m\times n}$, $x\in\R^n$.
Furthermore, we define the \emph{encoding length} or \emph{size} of an integer~$n\in\Z$ as
\begin{equation}
  \size{n} := 1 + \ceil{\log(\abs{n} + 1)}.
\end{equation}
Unless otherwise noted, logarithms throughout the paper are base two.
Encoding lengths of other objects are defined as follows.
For a rational number~$p/q$ with $p\in\Z$ and $q\in\nonneg\Z$, $\size{p/q}
:= \size{p} + \size{q}$.
For a vector~$v\in\Q^n$, $\size{v} := \sum_i\size{v_i}$ and for a
matrix~$A\in\Q^{m\times n}$, $\size{A} := \sum_{i,j}\size{A_{ij}}$.
Note that~$\size{v} \geq n$ and~$\size{A} \geq nm$.
To clearly distinguish between the size and the
value of numbers, we will often explicitly use the term \emph{value} when referring
to the numeric value taken by numbers.

\subsection{Approximate and Exact Solutions}

In order to compute exact rational solutions to linear programs, many
algorithms rely on a methodology of first computing sufficiently accurate
approximate solutions and then applying techniques to convert these solutions
to exact rational solutions.
This general technique is not unique to linear programming and has been applied
in other areas such as exact linear algebra.
In this section we will describe some results that make this possible.

First, we consider a result related to the \emph{Diophantine approximation problem},
a problem of determining low-denominator rational approximations $p/q$ of a
given number~$\alpha$.
\begin{theorem}[\cite{Schrijver1986}, Cor.~6.3b]\label{the:diophantine}
  For $\alpha\in\Q$, $M > 0$ there exists at most one rational number~$p/q$
  such that $\abs{p/q - \alpha} < 1/(2Mq)$ and $1 \leq q \leq M$.
  There exists a polynomial-time algorithm to test whether this number exists
  and, if so, to compute this number.
\end{theorem}
The proof makes use of continued fraction approximations.
The resulting algorithm, which is essentially the extended Euclidean algorithm, runs in
polynomial time in the encoding length of $\alpha$ and $M$ and is very fast in
practice.
This technique is sometimes referred to as ``rounding'' since for $M=1$ above it
corresponds to rounding to the nearest integer, or as ``numerical rational number
reconstruction'' (see \cite{WangPan2003}).  In the remainder of the paper we will
often refer to the process as simply \emph{rational reconstruction}.

Suppose there is an unknown rational number $p/q$ with $q\leq M$.
If an
approximation $\alpha$ that satisfies $\abs{p/q - \alpha} < 1/(2M^2)$ can be 
computed, then Theorem~\ref{the:diophantine} can be applied to determine
the exact value of $p/q$ in polynomial time.
Since basic solutions of linear programs correspond to solutions of linear
systems of equations, one can always derive an a priori bound on the size of the
denominators of any basic solution using Cramer's rule and the
Hadamard inequality as follows.
\begin{lemma}
  The entries of any basic primal-dual solution of a rational
  LP~\eqref{eq:standardlp} have denominator bounded by
  \begin{equation}\label{equ:hadamardlp}
    H := n^{m/2} L^n \prod_{j=1}^m \maxnorm{A_{j\cdot}}
  \end{equation}
  with~$L$ the least common multiple of the denominators of the entries
  in~$A,b,\lb$, and~$c$.
\end{lemma}
\begin{proof}
  Scaling with $L$ transforms \eqref{eq:standardlp} into an LP with identical
  primal-dual solutions $\min\{ (Lc)^\T x \mid (LA)x = Lb,~x \ge
  L\lb \}$.
  Then basic solutions are uniquely determined by linear systems with integral
  coefficient matrix $L\tilde B \in \Z^{n\times n}$, where $\tilde B$ is a square matrix
  constructed from all rows of $A$ plus unit vectors, see
  Equations~\eqref{equ:lpex:primalrowbasis} and \eqref{equ:lpex:dualrowbasis} of
  Section~\ref{subsec:proofs3}.
  By Cramer's rule, the denominators are then a factor of $|\det(L\tilde B)| = L^n |\det(\tilde B)|$.
  Applying Hadamard's inequality to the rows of $\tilde B$ yields $|\det(\tilde B)| \leq
  (\prod_{j=1}^m \norm{A_{j\cdot}}_2) \cdot 1 \cdots 1$.
  With $\norm{A_{j\cdot}}_2 \leq \sqrt{n} \maxnorm{A_{j\cdot}}$ we obtain the desired result $|\det(L\tilde B)|
  \leq n^{m/2} L^n \prod_{j=1}^m \maxnorm{A_{j\cdot}}$.
  \qed
\end{proof}

Therefore, upon computing an approximation of an optimal basic solution vector
where each component is within $1/(2H^2)$ of the exact value, we may apply
Theorem~\ref{the:diophantine} componentwise to recover the exact solution
vector. 

A different technique for reconstructing rational vectors, one that is
computationally more expensive but works under milder
assumptions, is based on polynomial-time \emph{lattice reduction} algorithms as
pioneered by~\cite{LenstraLenstraLovasz1982} and recently improved
by~\cite{NguyenStehle2009}.
It rests on the following theorem.

\begin{theorem}[\cite{LenstraLenstraLovasz1982}, Prop.~1.39]\label{the:simdiophantine}
  There exists a polynomial-time algorithm that, given positive integers~$n,M$
  and~$\alpha\in\Q^n$, finds integers~$p_1,\ldots,p_n,q$ for which
  \begin{align*}
    \abs{p_i/q - \alpha_i} &\leq 1/(Mq) \text{ for } i = 1,\ldots,n,\\
    1 \leq q &\leq 2^{n(n+1)/4} M^n.
  \end{align*}
\end{theorem}

Note that in contrast to Theorem~\ref{the:diophantine}, the above can be used to
recover the entire solution vector at once, instead of componentwise;
this process is often referred to as \emph{simultaneous Diophantine approximation}.
This has prominently been applied by~\cite{GroetschelLovaszSchrijver1988} in
the following fashion.
Let the \emph{facet complexity} of a polyhedron~$P$ be the smallest 
number~$\phi \geq n$ such that $P$ is defined by a list of inequalities 
each having encoding length at most~$\phi$.

\begin{lemma}[\cite{GroetschelLovaszSchrijver1988}, Lem.~6.2.9]\label{the:simrounding}
  Suppose we are given a polyhedron~$P$ in~$\R^n$ with facet complexity at
  most~$\phi$ and a point~$\alpha\in\R^n$ within Euclidean distance at
  most~$2^{-6n\phi}$ of~$P$.
  If~$p \in \Z^n$ and $q \in \Z$ satisfy
  \begin{align}\label{equ:simrounding}
  \begin{aligned}
    \norm{p - q \alpha}_2 &\leq 2^{-3\phi} \text{ and}\\
    1 \leq q &< 2^{4n\phi}
  \end{aligned}
  \end{align}
  then~$\tfrac{1}{q} p$ is in~$P$.
\end{lemma}

A solution~$\tfrac{1}{q} p$ satisfying~\eqref{equ:simrounding} can be computed
in polynomial time by the lattice reduction algorithm behind
Theorem~\ref{the:simdiophantine}.
Current lattice reduction algorithms are, despite being polynomial-time, generally
slower than the continued-fraction based techniques discussed above.
When applying Theorem~\ref{the:diophantine} componentwise to reconstruct a vector
there exist heuristic methods to accelerate this process by taking advantage of 
the fact that denominators of the component vectors often share common factors~\cite{ChenStorjohann2005,CookSteffy2011}.
We now provide an overview of algorithms for computing exact solutions to 
rational linear programs.

\subsection{Methods for Exact Linear Programming}

Exact polynomial-time algorithms for solving rational LPs based on the ellipsoid
method of Khachiyan \cite{Khachiyan1980} are described
in Gr\"otschel et al.~\cite{GroetschelLovaszSchrijver1988}.
There, a clear distinction is made between the \emph{weak problem} of finding approximate 
solutions and the \emph{strong problem} of finding exact solutions.
The ellipsoid method produces smaller and smaller ellipsoids enclosing an optimal
solution such that eventually simultaneous Diophantine approximation can be
applied to recover an exact rational solution from the center of the ellipsoid.
The same methods could equally be applied in the context of interior point
algorithms in order to convert an approximate, sufficiently advanced solution
along the central path to an exact rational solution.
However, the original algorithm of Karmarkar~\cite{Karmarkar1984} moves from the
approximate solution to a nearby basis solution that matches as closely as
possible and checks this for optimality.
The variant of Renegar~\cite{Renegar1988} instead identifies an optimal face of the
feasible region from approximately tight inequalities and performs a projection
step via solving a system of linear (normal) equations.
The paper includes a detailed analysis and discussion of these ideas and also references
the method developed by Edmonds~\cite{Edmonds1967} for solving linear systems of
equations in polynomial time.
(Note that if one is not careful, Gaussian elimination may require an exponential number
of bit operations.)
While these methods exhibit polynomial worst-case complexity, the high levels of
precision required may be too high to use in any practical setting.
This is illustrated by the following example.

\begin{example}
Consider the following small and unremarkable LP.
\begin{align*}
 ~~\max ~2x_1& +3x_2   +2x_3+ ~x_4+ 2x_5- ~x_6\\
           s.t.~~ x_1& + ~x_2~ + 2x_3+ 3x_4  + ~x_5   ~~ ~ \qquad         \leq 3\\
                     x_1& - ~x_2    \quad \qquad      + ~x_4+3 x_5 - 2x_6 \leq 2\\
                     x_1& + 2x_2 + ~x_3+ 3x_4\qquad \quad   +~x_6\leq 4\\
                     x_1&, x_2, x_3, x_4, x_5, x_6 \geq 0
\end{align*}
Theorem 6.3.2 of Gr\"otschel et al.~\cite{GroetschelLovaszSchrijver1988} says that
an exact rational solution to an LP can be found by first calling a weak optimization
oracle to find an approximate solution to the LP, and then apply simultaneous
Diophantine approximation to recover the exact rational solution.
The tolerance indicated for this purpose is given as
$\epsilon = \frac{2^{-18n\langle c \rangle - 24n^4 \phi}}{\|c\|_\infty}$, where $\phi$ is the
facet complexity of the feasible region and $c$ is the objective vector.
For the above problem, this works out to $\epsilon \approx 10^{-169,059}$.
In comparison, double-precision arithmetic only stores about 16 significant decimal digits.
For any problem of real practical interest, $\epsilon$ will be even smaller.
This underlines that---despite the
fact that the $\epsilon$ is suitable for establishing polynomial running time of algorithms---
it may be far beyond what is feasible for real computations in practice.

By contrast, the largest encoding length of any vertex of the above example is merely 27 and the largest
denominator across all vertices is 8.
Thus, a solution whose componentwise difference from a vertex was under 1/128 would
be sufficient to apply Theorem~\ref{the:diophantine} componentwise to recover the vertex.
\end{example}

Also in general, many LPs of practical interest are highly sparse and may have other special characteristics
that result in their solutions having encoding length dramatically smaller than the value of
derivable worst-case bounds on these values.
Therefore it is of interest to work with what are often known as \emph{output sensitive} algorithms;
where the running time on a problem instance depends not only on the input length, but also on the
size of the output.
Many of the algorithms described in the remaining literature review, and the results derived in this paper,
can be thought of in this context as they often have the chance to find an exact solution and terminate
earlier than a worst-case bound might suggest.

The remainder of this subsection focuses on methods used in practice for
computing exact rational solutions to linear programs.
We first note that many of these practical methods are based on the simplex
method.
There is a trivial method of solving LPs with rational input data exactly,
which is to apply a simplex algorithm and perform all computations in (exact) 
rational arithmetic.
Espinoza~\cite{Espinoza2006} observed computationally that the running time
of this na\"ive approach depends heavily on the encoding length of the rational
numbers encountered during intermediate computations and is much slower than
floating-point simplex implementations.

Edmonds noted that the inverse of a basis matrix can be represented as matrix of integer
coefficients divided by a common denominator and that this representation can be
efficiently updated when pivoting from basis to basis;
this method is referred to as the $Q$-method and is further developed
by~\cite{EdmondsMaurras1997} and~\cite{AzulayPique1998}.
Compared to computing a basis inverse with rational coefficients, the $Q$-method
avoids the repeated GCD computations required by exact rational arithmetic.
Escobedo and Moreno-Centeno~\cite{EscobedoMorenoCenteno2015,EscobedoMorenoCenteno2017}
have applied similar ideas to achieve roundoff-error free matrix factorizations and updates.

The most successful methods in practice for solving LPs exactly over the rational
numbers rely on combining floating-point and exact arithmetic in various ways.
The key idea is to use the the inexact computation in ways that can provide
useful information, but that exact computation is used for critical decisions
or to validate final solutions.
For example, G\"artner~\cite{Gaertner1999} uses floating-point arithmetic for some parts
of the simplex algorithm, such as pricing, but uses exact arithmetic and ideas from
Edmonds' $Q$-method when computing the solutions.
Another way to combine inexact and exact computation relies on the observation
that floating-point solvers often return optimal bases; since the basis provides
a structural description of the corresponding solution, it can be recomputed
using exact arithmetic and checked for optimality via LP duality
\cite{DhiflaouiEtAl2003,Koch2004,Kwappik1998}.
This approach was systematized by Applegate et al.~\cite{ApplegateCookDashEspinoza2007},
in the \qsoptex solver which utilizes increasing levels of arithmetic precision
until an optimal basis is found, and its rational solution computed and verified.
This approach of \emph{incremental precision boosting} is often very effective at
finding exact solutions quickly, particularly when the initial double-precision
LP subroutines are able to find an optimal LP basis, but becomes slower in cases
where many extended-precision simplex pivots are used.
In related work, Cheung and Cucker~\cite{CheungCucker2006} have developed and analyzed
exact LP algorithms that adopt variable precision strategies.
A recent algorithm that has proven most effective in practice for computing high-accuracy
and exact solutions to linear programs is \emph{iterative refinement} for linear
programming~\cite{GleixnerSteffyWolter2016}.
It will be described in detail in
Section~\ref{sec:ir} and serves as the basis for much of the work in this paper.

\subsection{Contribution and Organization of the Paper}

This paper explores the question of how LP oracles based on limited-precision
arithmetic can be used to design algorithms with polynomial running
time guarantees in order to compute exact solutions for linear programs with
rational input data.
In contrast to classic methods from the literature that rely on ellipsoid or
interior point methods executed with limited, but high levels of
extended-precision arithmetic, our focus is more practical, on oracles with low
levels of precision as used by standard floating-point solver implementations.
Section~\ref{sec:ir} formalizes this notion of limited-precision LP oracles and
revises the iterative LP refinement method from~\cite{GleixnerSteffyWolter2016}
in order to guarantee polynomial bounds on the encoding length of the numbers
encountered.
Sections~\ref{sec:oracle} and \ref{sec:ratrecon} present two methodologically
different extensions in order to construct exact solutions---basis verification
and rational reconstruction.
For both methods oracle-polynomial running time is established; for the latter,
we bound the running time by the encoding length of the final solution, which
renders it an output-sensitive algorithm.
Some of the more technical proofs are collected in Section~\ref{sec:proofs}.
Section~\ref{sec:lpex:experiments} analyzes the properties of both methods
computationally on a large test set of linear programs from public sources and
compares their performance to the incremental precision boosting algorithm
implemented in \qsoptex.
The code base used for the experiments is freely and publicly available for
research.
Concluding remarks are given in the final Section~\ref{sec:lpex:conclusion}.
 \section{Convergence Properties of Iterative Refinement}
\label{sec:ir}

Our starting point is the iterative refinement method proposed in
\cite{GleixnerSteffyWolter2016}, which uses calls to a limited-precision LP
solver in order to generate a sequence of increasingly accurate solutions.
Its only assumption is that the underlying LP oracle returns solutions with
absolute violations of the optimality conditions being bounded by a constant
smaller than one.
In the following we give a precise definition of a limited-precision LP oracle.
This formal notion is necessary to evaluate the behavior of the algorithms
defined in this paper, in particular to show that the number of oracle calls,
the size of the numbers encountered in the intermediate calculations, and the
time required for intermediate calculations are all polynomial in the encoding
length of the input.
It is also helpful to introduce the set~$\fpnumbers(\fpdigits) := \{ n / 2^\fpdigits \in
\Q : n\in\Z, \abs{n} \leq 2^{2\fpdigits} \}$ for some fixed~$p \in \mathbb{N}$;
this can be viewed as a superset of floating-point numbers, that is easier 
to handle in the subsequent proofs.
Standard IEEE-754 double-precision floating-point numbers, for example, are all
contained in~$\fpnumbers(1074)$.

\begin{definition}\label{def:oracle} We call an oracle a \emph{limited-precision LP oracle} if there exist constants~$\fpdigits \in \mathbb{N}$, $0 < \fptol < 1$, and $\sigma > 0$
such that for any LP
\begin{align}\tag{$P$}\label{prob:lpeq}
  \min\{ c^\T x : Ax = b, x \geq \lb \}
\end{align}
with~$A\in\Q^{m\times n}$, $b\in\Q^m$, and $c,\lb\in\Q^n$, the oracle either reports a ``failure'' or returns an approximate primal--dual
solution~$\bar{x}\in\fpnumbers(\fpdigits)^n$, $\bar{y}\in\fpnumbers(\fpdigits)^m$ that satisfies\begin{subequations}
\begin{eqnarray}
  & \norm{A\bar x - b}_\infty \leq \fptol,\\
  & \bar x \geq \lb - \fptol\one,\\
  & c - A^\T \bar y \geq -\fptol\one,\\
  & \abs{(\bar{x} - \lb)^\T (c - A^\T \bar{y})} \leq \sigma,
\end{eqnarray}
\end{subequations}
when it is given the LP
\begin{align}\tag{$\bar P$}
  \min\{ \bar c^\T x : \bar Ax = \bar b, x \geq \bar\lb \}
\end{align}
where~$\bar A \in\Q^{m\times n}$, $\bar c, \bar \lb\in\Q^n$, and $\bar b\in\Q^m$
are $A,c,\lb$, and~$b$ with all numbers rounded to~$\fpnumbers(\fpdigits)$.
We call the oracle a \emph{limited-precision LP-basis oracle} if it additionally returns
a basis $\B \subseteq \{1,\ldots,n\}$ satisfying
\begin{subequations}
\begin{eqnarray}
  & \abs{\bar x_i - \lb_i} \leq \fptol \text{ for all } i\not\in\B,\label{def:oracle:primalbasis}\\
  & \abs{c_i - \bar y^\T A_{\cdot i}} \leq \fptol \text{ for all } i\in\B,\label{def:oracle:dualbasis}
\end{eqnarray}
\end{subequations}
\end{definition}

Relating this definition with the behavior of real-world limited-precision LP solvers,
we note that real-world solvers are certainly not guaranteed to find a solution with residual
errors bounded by a fixed constant.\footnote{  Otherwise, as noted in~\cite{GleixnerSteffyWolter2016}, an oracle that guarantees a
  solution with residual errors bounded by a fixed constant could be queried for solutions
  of arbitrarily high precision in a single call by first scaling the problem data by a
  large constant. As we will see, this type of manipulation is not used by our algorithms. 
  Moreover, the requirement in the definition of the oracle that the input data entries
  are in the bounded set $\fpnumbers(\fpdigits)$ would forbid the input of arbitrarily scaled LPs.}
However, these errors could nonetheless be computed
and checked, correctly identifying the case of ``failure''.
Algorithm~\ref{algo:sac} states the basic iterative refinement procedure
introduced in \cite{GleixnerSteffyWolter2016}.
For clarity of presentation, in contrast to \cite{GleixnerSteffyWolter2016},
Algorithm~\ref{algo:sac} uses equal primal and dual scaling factors and tracks
the maximum violation of primal feasibility, dual feasibility, and complementary
slackness in the single parameter~$\delta_k$.
The basic convergence result, restated here as Lemma~\ref{lem:convrate}, carries over from \cite{GleixnerSteffyWolter2016}.

\begin{lemma}\label{lem:convrate} Given an LP of form \eqref{prob:lpeq} and a limited-precision LP oracle with
constants~$\fptol$ and $\sigma$, let~$x_k,y_k,\scalfac_{k}$,
$k=1,2,\ldots$, be the sequences of primal--dual solutions and scaling factors
produced by Algorithm~\ref{algo:sac} with incremental scaling limit~$\alpha \geq 2$. Let~$\epsilon := \max\{\fptol, 1/\alpha\}$.
Then for all iterations~$k$, $\scalfac_{k+1} \geq \scalfac_{k} / \epsilon$, and
\begin{subequations}
\begin{eqnarray}
  &  \norm{Ax_k - b}_\infty \leq \epsilon^k,\label{point1}\\
  &  x_k - \lb \geq -\epsilon^k\one,\label{point2}\\
  &  c - A^\T y_k \geq -\epsilon^k\one,\label{point3}\\
  &  \abs{(x_k - \lb)^\T (c - A^\T y_k)} \leq \sigma\epsilon^{2(k-1)}.
\end{eqnarray}
\end{subequations}
Hence, for any $\termtol > 0$, Algorithm~\ref{algo:sac} terminates in finite time after at most
\begin{equation}\label{lem:convrate:maxiters}
  \ceil{\max\{ \log(\termtol) / \log(\epsilon),
  \log(\termtol\epsilon/\sigma) / \log(\epsilon^2) \}}
\end{equation}
calls to the limited-precision LP oracle.
\end{lemma}

\begin{proof}
  Corollary~1 in \cite{GleixnerSteffyWolter2016} proves the result for a more
  general version of Algorithm~\ref{algo:sac} that treats primal and dual scaling factors
  independently, but does not include the rounding step in
  line~\ref{algo:sac:scaling2}.
  However, the same proof continues to hold because the upward rounding does not decrease the scaling factor~$\scalfac_{k+1}$.
  Hence, the termination criterion $\delta_k \leq \termtol$ in
  line~\ref{algo:sac:termination} is met if $\max\{ \epsilon^k,
  \sigma\epsilon^{2(k-1)} \} \leq \termtol$, which holds when $k$ reaches the
  bound in \eqref{lem:convrate:maxiters}.
  Note that this bound on the number of refinements also holds in case the
  algorithm aborts prematurely after a failed oracle call.
  \qed
\end{proof}

\begin{algorithm}[t]
\caption{Iterative Refinement for a Primal and Dual Feasible LP}
\label{algo:sac}

\vspace*{0.5ex}

\Input{rational LP data  $A,b,\lb,c$,
  termination tolerance $\termtol \geq 0$}

\Params{incremental scaling limit $\alpha \in \mathbb{N}, \alpha \geq 2$}

\Output{primal--dual solution $x^*\in\Q^n,y^*\in\Q^m$ within tolerance $\tau$}

\Begin{
    $\scalfac_{1} \assign 1$\tcc*{initial solve}
    get $(\bar A,\bar b, \bar \lb, \bar c) \approx (A,b,\lb,c)$ in working precision of the oracle\;
    call oracle for $\min\{ \bar c^\T x : \bar Ax = \bar b, x \geq \bar\lb \}$, abort if failure\label{algo:sac:initfpsolve}\;
    $(x_{1},y_{1})$ \assign  approximate primal--dual solution returned\label{algo:sac:firstsol}\;
    \BlankLine
    \For(\tcc*[f]{refinement loop}){$k \assign 1,2,\ldots$}{
                        $\hat b \assign b - A x_k$,
      $\hat \lb \assign \lb - x_k$,
      $\hat c \assign c - A^\T y_k$\tcc*{compute residual error}

                              $\delta_{k} \assign \max\left\{ \max_{j} | \hat b_j |,
      \max_{i} \hat\lb_i, \max_i -\hat c_i,
      \abs{\sum_{i} -\hat\lb_i \hat c_i} \right\}$\label{algo:sac:violation}\;

      \lIf{$\delta_{k} \leq \termtol$\label{algo:sac:termination}}{
        return $x^* \assign x_k, y^* \assign y_k$
      }

            $\delta_{k} \assign \max\{ \delta_{k}, 1 / (\alpha\scalfac_{k}) \}$\label{algo:sac:scaling1}\tcc*{scale problem}
      $\scalfac_{k+1} \assign 2^{\ceil{\log (1/\delta_{k})}}$\label{algo:sac:scaling2}\tcc*{round scaling factor to a power of two}
      get $(\bar b, \bar \lb, \bar c) \approx \scalfac_{k+1} (\hat b, \hat \lb, \hat c)$ in working precision of the oracle\label{algo:sac:rounding}\;
            \tcc*{solve for corrector solution}
      call oracle for $\min\{ \bar c^\T x : \bar Ax = \bar b, x \geq \bar\lb \}$, abort if failure\label{algo:sac:reffpsolve}\;
      $(\hat x,\hat y)$ \assign approximate primal--dual solution returned\label{algo:sac:corrsol}\;

                        $(x_{k+1},y_{k+1}) \assign (x_{k},y_{k}) + \frac{1}{\scalfac_{k+1}} (\hat x,\hat y)$\label{algo:sac:correct}\tcc*{perform correction}
  }
}
\end{algorithm}
 
This lemma proves that the number of calls to the LP oracle before reaching a
positive termination tolerance~$\termtol$ is linear in the encoding length of~$\termtol$.
However, the correction step in line~\ref{algo:sac:correct}
could potentially cause the size of the numbers
in the corrected solution to grow exponentially.
The size of~$x_{k+1}$ and~$y_{k+1}$ may be as large as~$2 (\size{x_k} +
\size{\scalfac_{k+1}} \size{\hat x})$ and~$2 (\size{y_k} + \size{\scalfac_{k+1}}
\size{\hat y})$, respectively, if the numbers involved have arbitrary
denominators.
The following lemma shows that the rounding of the scaling factors prevents this
behavior.

\begin{lemma}\label{lem:polygrowth}
  The size of the numbers encountered during Algorithm~\ref{algo:sac} with a
  limited-precision LP oracle according to Definition~\ref{def:oracle} is
  polynomially bounded in the size of the input $A,b,\lb,c,\termtol$, when $\termtol>0$.
\end{lemma}

A technical proof is provided in Section~\ref{subsec:proofs2}.
The bound established on the encoding length of all numbers encountered during
the course of the algorithm is $\order(\size{A,b,\lb,c}+ (n+m)\size{\termtol})$.
This leads to the main result of this section.

\begin{theorem}\label{the:lpir:polytime}
  Algorithm~\ref{algo:sac} with a limited-precision LP oracle according to
  Definition~\ref{def:oracle} runs in oracle-polynomial time, i.e., it requires a polynomial number of oracle calls and a polynomial number of bit operations in the size of the input $A,b,\lb,c,\termtol$, when $\termtol>0$.
\end{theorem}

\begin{proof}
  The initial setup and each iteration are~$\order(n+m+nnz)$ operations on numbers
  that, by Lemma~\ref{lem:polygrowth}, are of polynomially bounded size.
  Here $nnz$ denotes the number of nonzero entries in~$A$.
  By Lemma~\ref{lem:convrate} we know that the number of iterations of the
  algorithm is polynomially bounded in the encoding length of the input.
  \qed
\end{proof}

 \section{Oracle Algorithms with Basis Verification}
\label{sec:oracle}

Iterative refinement as stated in Algorithm~\ref{algo:sac} only terminates in
finite time for positive termination tolerance~$\termtol > 0$.
The first extension, presented in this section, assumes a
\emph{limited-precision LP-basis oracle} as formalized in
Definition~\ref{def:oracle} and computes exact basic solutions with zero
violation in finite, oracle-polynomial running time.

\subsection{Convergence of Basic Solutions}

If the LP oracle additionally returns basic solutions for the transformed problem $\min\{
\hat c^\T x : \hat Ax = \hat b, x \geq \hat\lb \}$, then
Algorithm~\ref{algo:sac} produces a sequence $(x_k,y_k,\B_k)_{k=1,2,\ldots}$.
Theorem 3.1 in~\cite{GleixnerSteffyWolter2016} already states that if the
corrector solution~$\hat x,\hat y$ returned by the LP oracle is the exact basic
solution for $\B_k$, then the corrected solution~$x_{k+1},y_{k+1}$ in
line~\ref{algo:sac:correct} is guaranteed
to be the unique solution to the original LP determined by $\B_k$.
This is only of theoretical interest since the LP oracle returns only
approximately basic solutions:
Still, we can ask whether and under which conditions the sequence of bases is
guaranteed to become optimal in a finite number of refinements.
We show that properties~\eqref{def:oracle:primalbasis} and
\eqref{def:oracle:dualbasis} suffice to guarantee optimality after a number of
refinements that is polynomial in the size of the input.
The proof relies on the fact that there are only finitely many non-optimal basic
solutions and that their infeasibilities are bounded.
This is formalized by the following lemma.

\begin{lemma}\label{lem:lpex:basisviolation}   Given an LP \eqref{prob:lpeq} with rational data~$A\in\Q^{m\times n}$, $b\in\Q^m$, and~$\lb,c\in\Q^n$,
  the following hold for any basic primal--dual solution~$x,y$:
  \begin{enumerate}
  \item Either~$x$ is (exactly) primal feasible or its maximum primal violation
    has at least the value $1 / 2^{4\size{A,b} + 5\size{\lb} + 2n^2 + 4n}$.
  \item Either~$y$ is (exactly) dual feasible or its maximum dual violation has
    at least the value~$1 / 2^{4\size{A, c} + 2n^2 + 4n}$.
  \end{enumerate}
\end{lemma}

  A detailed proof of Lemma~\ref{lem:lpex:basisviolation} is found in
  Section~\ref{subsec:proofs3} but the basic idea is summarized as follows.
  Suppose~$x,y$ is a basic primal--dual solution with respect to\ some basis~$\B$.
  By standard arguments, we show that the size of the entries in $x$ and
  $y$ is bounded by a polynomial in~$\size{A,b,\lb,c}$ and that all possible
  violations can be expressed as differences of rational numbers with bounded
  denominator (or zero).

The following theorem states the main convergence result.

\begin{theorem}\label{the:lpex:optconv}   Suppose we are given an LP \eqref{prob:lpeq},
  a fixed $\epsilon$, $0 < \epsilon < 1$, and a sequence of primal--dual solutions~$x_k, y_k$ with associated bases~$\B_k$ such that (\ref{point1}--\ref{point3}) and
  \begin{subequations}
    \begin{eqnarray}
      &  \abs{(x_k)_i - \lb_i} \leq \epsilon^k \text{ for all } i\not\in\B_k,\label{point4}\\
      &  \abs{c_i - y_k^\T A_{\cdot i}} \leq \epsilon^k \text{ for all } i\in\B_k\label{point5}
    \end{eqnarray}
  \end{subequations}
  hold for $k=1,2,\ldots$.
  Then there exists a threshold $K=K(A,m,n,b,\lb,c,\epsilon)$ such that the bases $\B_k$ are optimal for all~$k \geq K$.
  The function satisfies the asymptotic bound
  $K(A,m,n,b,\lb,c,\epsilon) \in \order((m^2\size{A} + \size{b,\lb,c} + n^2)/ \log(1/\epsilon))$.
\end{theorem}

\noindent
  Again, the detailed proof of Theorem~\ref{the:lpex:optconv} is found in
  Section~\ref{subsec:proofs3}.
  The proof uses \eqref{point4} and \eqref{point5} to
  show analogs of~(\ref{point1}--\ref{point3}) hold with right-hand side~$2^{4m^2\size{A} + 2} \epsilon^k$
  for the solutions~$\tilde x_k, \tilde y_k$ associated with bases $\B_k$.
  Then for~$k \geq K$, the primal and dual
  violations of~$\tilde x_k, \tilde y_k$ drop below the minimum
  thresholds stated in Lemma~\ref{lem:lpex:basisviolation}.  From then
  on~$\B_k$ must be optimal.

  Conditions~(\ref{point1}--\ref{point3}) require that the
  primal and dual violations of~$x_k, y_k$ converge to zero precisely as is
  guaranteed by Lemma~\ref{lem:convrate} for the sequence of numeric solutions
  produced by iterative refinement.
  Additionally, \eqref{point4} and \eqref{point5} assume
  that the numeric solutions become ``more and more basic'' in the sense that
  the deviation of the nonbasic variables from their bounds and the absolute
  value of the reduced costs of basic variables converges to zero at the same
  rate as the primal and dual violations.
This is shown by the following lemma using properties~\eqref{def:oracle:primalbasis}
and~\eqref{def:oracle:dualbasis} of Definition~\ref{def:oracle}.

\begin{lemma} \label{the:lpex:basconv}
Suppose we are given a primal and dual feasible LP \eqref{prob:lpeq} and a
limited-precision LP-basis oracle according to Definition~\ref{def:oracle} with
constants~$\fpdigits$, $\fptol$, and $\sigma$.
Let $x_k,y_k,\B_k,\scalfac_{k}$, $k=1,2,\ldots$, be the
sequences of primal--dual solutions, bases, and scaling factors produced by
Algorithm~\ref{algo:sac} with incremental scaling limit~$\alpha \geq 2$, and
let~$\epsilon := \max\{\fptol, 1/\alpha\}$.
Then conditions~\eqref{point4} and~\eqref{point5} are satisfied for all~$k$.
\end{lemma}

\begin{proof}
  We prove both points together by induction over $k$.
  For $k=1$, they hold directly because the defining conditions
  \eqref{def:oracle:primalbasis} and \eqref{def:oracle:dualbasis} are satisfied
  for the initial floating-point solution and $\fptol \leq \epsilon$.
  Suppose~\eqref{point4} and~\eqref{point5} hold for $k \geq 1$ and consider
  $k+1$.
  Let $\hat x,\hat y, \hat\B$ be the last approximate solution returned by the
  LP solver.
  Then for all~$i\not\in\B_{k+1}$
  \begin{align*}
   \abs{(x_{k+1})_i - \lb_i}
   &= \bigabs{((x_k)_i + \frac{\hat x_i}{\scalfac_{k+1}}) - \lb_i}\\
   &= \abs{\hat x_i + \scalfac_{k+1}( \underbrace{(x_k)_i - \lb_i}_{= -\hat\lb_i} )} / \underbrace{\scalfac_{k+1}}_{\mathclap{\hspace*{3em}\geq \epsilon^{-k} \text{ by Lemma~\ref{lem:convrate}}}}\\
   &\leq \abs{\hat x_i - \scalfac_{k+1} \hat\lb_i} / \epsilon^k
   \stackrel{\eqref{def:oracle:primalbasis}}{\leq} \fptol\epsilon^k
   \leq \epsilon^{k+1},
  \end{align*}
  and similarly for all~$i\in\B_{k+1}$
  \begin{align*}
   \abs{c_i - y_{k+1}^\T A_{\cdot i}}
   &= \bigabs{c_i - \big( y_{k} + \frac{\hat y}{\scalfac_{k+1}} \big)^\T A_{\cdot i}}\\
   &= \abs{\scalfac_{k+1}( \underbrace{c_i - y_k^\T A_{\cdot i}}_{= \hat c_i} ) - \hat y^\T A_{\cdot i}} / \underbrace{\scalfac_{k+1}}_{\mathclap{\hspace*{3em}\geq \epsilon^{-k} \text{ by Lemma~\ref{lem:convrate}}}}\\
   &\leq \abs{\scalfac_{k+1}\hat c_i - \hat y^\T A_{\cdot i}} / \epsilon^k
   \stackrel{\eqref{def:oracle:primalbasis}}{\leq} \fptol\epsilon^k
   \leq \epsilon^{k+1}.
  \end{align*}
  This completes the induction step.
  \qed
\end{proof}

\subsection{Iterative Refinement with Basis Verification}
\label{sec:lpex:ratfac:basver}

  The bound on the number of refinements may seem surprisingly large, especially when
  considering that the best-known iteration complexity for interior point
  methods is\linebreak $\order(\sqrt{n + m} \size{A,b,\lb,c})$~\cite{Renegar1988} and
  that in each refinement round we solve an entire LP.
  One reason for this difference is that iterative refinement converges only
  linearly as proven in Lemma~\ref{lem:convrate}, while interior point
  algorithms are essentially a form of Newton's method, which allows for
  superlinear convergence.
  Additionally, the low-precision LPs solved by Algorithm~\ref{algo:sac} may be 
  less expensive in practice than performing interior point iterations in very high-precision arithmetic.
  Nevertheless, the convergence results above provide an important
  theoretical underpinning for the following algorithm.

As already observed experimentally in~\cite{GleixnerSteffyWolter2016}, in
practice, an optimal basis is typically reached after very few refinements, much
earlier than guaranteed by the worst-case bound of
Theorem~\ref{the:lpex:optconv}.
Hence, we do not want to rely on bounds computed a priori, but check the
optimality of the basis early.
A natural idea is to perform an exact rational solve as soon as the basis
has not changed for a specified number of refinements.
This is easily achieved by extending Algorithm~\ref{algo:sac} as follows.

Suppose the LP oracle returns a basis~$\hat\B$ in line~\ref{algo:sac:corrsol}.
First, we continue to perform the quick correction step and check the
termination conditions for the corrected solution until
line~\ref{algo:sac:termination}.
If they are violated, we solve the linear systems of equations associated
with~$\hat\B$ in order to obtain a basic solution $\tilde x,\tilde y$.
Because it is by construction complementary slack, we only check primal and dual
feasibility.
If this check is successful, the algorithm terminates with $x^* = \tilde x$ and
$y^* = \tilde y$ as optimal solution.
Otherwise it is discarded and Algorithm~\ref{algo:sac} continues with the next
refinement round.

In practice, the basis verification step can be skipped if the basis has not
changed since the last LP oracle call in order to save redundant computation.
Furthermore, because the linear systems solves can prove to be expensive, in
practice, it can be beneficial to delay them until the LP oracle has returned
the same basis information for a certain number of refinement rounds.
This does not affect the following main result.

\begin{theorem}   Suppose we are given a rational, primal and dual feasible LP
  \eqref{prob:lpeq} and a limited-precision LP-basis oracle according to
  Definition~\ref{def:oracle}.
  Algorithm~\ref{algo:sac} interleaved with a basis verification step before
  Line~\ref{algo:sac:scaling1} as described above terminates with an
  optimal solution to \eqref{prob:lpeq} in oracle-polynomial running time.
\end{theorem}

\begin{proof}
  Lemma~\ref{lem:convrate} and Lemma~\ref{the:lpex:basconv} prove
  that the sequence of basic solutions~$x_k,y_k,\B_k$ satisfies the conditions
  of Theorem~\ref{the:lpex:optconv}, hence~$\B_k$ is optimal after a polynomial
  number of refinements.
  According to Theorem~\ref{the:lpir:polytime}, this runs in oracle-polynomial
  time.
  As proven by~\cite{Edmonds1967}, the linear systems used to compute the
  basic solutions exactly over the rational numbers can be solved in polynomial
  time and this computation is done at most once per refinement round.
  \qed
\end{proof}

 \section{Rational Reconstruction Algorithms}
\label{sec:ratrecon}

The LP algorithm developed in the previous section relies solely on the
optimality of the basis information to construct an exact solution.
Except for the computation of the residual vectors it does not make use of the
more and more accurate numerical solutions produced.
In this section, we discuss a conceptually different technique that
exploits the approximate solutions as starting points in order to reconstruct from them
an exact optimal solution.
First we need to show that the sequence of approximate solutions actually converges.

\subsection{Convergence to an Optimal Solution}

Until now the convergence of the residual errors to zero was sufficient for our
results and we did not have to address the question whether the sequence of
solutions $x_k,y_k$ itself converges to a limit point.
The following example shows that this does not necessarily hold if the solutions
returned by the LP oracle are not bounded.

\begin{example}
  Consider the degenerate LP
  \begin{align*}
    \min\{ x_1 - x_2 \mid x_1 - x_2 = 0, -x_1 + x_2 = 0, x_1, x_2 \geq 0 \}.
  \end{align*}
  One can show that Algorithm~\ref{algo:sac} may produce the sequence of primal--dual
  solutions
  \begin{align*}
    x_k &= \big( 2^{k} + 2^{-k-1}, 2^{k} - 2^{-k-1} \big)\\
    y_k &= \big( 2^{k} + 2^{-1} + 2^{-3k-1}, 2^{k} - 2^{-1} - 2^{-3k-1} \big)
  \end{align*}
  for~$k=1,2,\ldots$.
  This happens if the LP oracle returns the approximate solution~$\hat x_k =
  \big( 2^{2k} - 1/4,2^{2k} + 1/4 \big)$ and~$\hat y_k = \big( 2^{2k} - 7\cdot
  2^{-2k-4},2^{2k} + 7\cdot 2^{-2k-4} \big)$ for the $k$-th transformed problem
  \begin{align*}
    \min\{ - 1/2^{2k} x_1 + 1/2^{2k} x_2 \mid\, &x_1 - x_2 = -1,
    -x_1 + x_2 = +1,\\
    &x_1 \geq -2^{2k-1} - 1/2,
    x_2 \geq -2^{2k-1} + 1/2 \},
  \end{align*}
  which is obtained with scaling factors~$\scalfac_{k} = 2^{k}$.
                                The primal violation of $x_k,y_k$ is $2^{-k}$, the dual violation is~$2^{-3k}$, and the
  violation of complementary slackness is~$2^{-4k}$.  Hence, all residual errors go to zero, but the
  iterates themselves go to infinity.
\end{example}

However, for corrector solutions from limited-precision
LP oracles the following holds.

\begin{lemma}\label{lem:cauchy}
  Given a rational, primal and dual feasible LP \eqref{prob:lpeq}
  and a limited-precision LP-basis oracle with precision~$\fpdigits$, let $(x_k,y_k,\scalfac_{k})_{k=1,2,\ldots}$ be the sequence of primal--dual
  solutions and scaling factors produced by Algorithm~\ref{algo:sac}.
  Define $C := 2^\fpdigits$.
  Then $(x_k,y_k)$ converges to a rational, basic, and optimal solution
  $(\tilde{x},\tilde{y})$ of \eqref{prob:lpeq} such that
  \begin{equation}
    \maxnorm{(\tilde{x},\tilde{y}) - (x_k,y_k)} \leq C
    \sum_{i=k+1}^{\infty} \scalfac_i^{-1}.
  \end{equation}
\end{lemma}

\begin{proof}
  This result inherently relies on the boundedness of the corrector solutions
  returned by the oracle.
  Since their entries are in $\fpnumbers(p)$, $\maxnorm{(\hat x_k,\hat y_k)}
  \leq 2^\fpdigits$.
  Then $(x_k,y_k) = \sum_{i = 1}^k \frac{1}{\scalfac_{i}} (\hat x_i,\hat y_i)$
  constitutes a Cauchy sequence: for any~$k,k' \geq K$,
  \begin{align}
    \maxnorm{(x_k,y_k) -
    (x_{k'},y_{k'})} \leq 2^\fpdigits \sum_{i = K+1}^\infty \epsilon^i =
    2^\fpdigits \epsilon^{K+1} / (1 - \epsilon),
  \end{align}
  where~$\epsilon$ is the rate of
  convergence from Lemma~\ref{lem:convrate}.
  Thus, a unique limit point~$(\tilde{x},\tilde{y})$ exists.

  The fact that $(\tilde{x},\tilde{y})$ is basic, hence also rational, follows
  from Claims~1 and~3 in the proof of Theorem~\ref{the:lpex:optconv}, see
  Section~\ref{sec:proofs}.
  Because of $\maxnorm{(\tilde{x},\tilde{y}) - (\tilde x_k,\tilde y_k)} \leq
  \maxnorm{(\tilde{x},\tilde{y}) - (x_k,y_k)} + \maxnorm{(x_k,y_k) - (\tilde
    x_k,\tilde y_k)}$, also the sequence of basic primal--dual solutions
  $(\tilde x_k,\tilde y_k)$ converges to the limit
  point~$(\tilde{x},\tilde{y})$.
  Since there are only finitely many basic solutions, this implies that
  $(\tilde{x},\tilde{y})$ must be one of the $(\tilde x_k,\tilde y_k)$.
  \qed
\end{proof}

Note that the statement holds for any upper bound $C$ on the absolute values in
the corrector solutions $\hat{x},\hat{y}$ returned by the oracle in
line~\ref{algo:sac:corrsol} of Algorithm~\ref{algo:sac}.
In practice, this may be much smaller than the largest floating-point
representable value, $2^\fpdigits$.

\subsection{Output-Sensitive Reconstruction of Rational Limit Points}

Suppose we know \emph{a priori} a bound~$M$ on the denominators in the limit,
then we can compute~$\tilde{x},\tilde{y}$ from an approximate solution
satisfying~$\maxnorm{(x_k,y_k) - (\tilde{x},\tilde{y})} < 1/(2M^2)$ by applying Theorem~\ref{the:diophantine} componentwise.
If the size of~$M$ is small, i.e., polynomial in the input size, then iterative
refinement produces a sufficiently accurate solution after a polynomial number
of refinements.
This eliminates the need to use other
methods such as rational matrix factorization to compute an exact solution.

However, known worst-case bounds for denominators in basic solutions are often very weak.
This has been demonstrated by~\cite{AbbottMulders2001} for random square
matrices and by~\cite{Steffy2011} for the special case of selected basis matrices from
linear programs.
For the Hadamard bound~$H$ from~\eqref{equ:hadamardlp} that holds for \emph{all} basis
matrices of an LP, this situation must be even more pronounced.
Tighter bounds that are reasonably cheap to compute are---to the best of our
knowledge---not available.
Computing an approximate solution with error below~$1/(2H^2)$ before applying the
rounding procedure can thus be unnecessarily expensive.
This motivates the following design of an output-sensitive algorithm, Algorithm~\ref{algo:lpex:ratrec}, that attempts to reconstruct exact
solution vectors during early rounds of refinement and tests the correctness of these
heuristically reconstructed solutions exactly using rational arithmetic.
We now give a description of it, followed by a proof of correctness and 
analysis of its running time.

The algorithm is an extension of the basic iterative refinement for linear programs,
Algorithm~\ref{algo:sac}, interleaved with attempts at rational reconstruction.
For~$k=1$, the algorithm starts with the first oracle call to obtain
the initial approximate solution and the corresponding residual errors.
Unless the solution is exactly optimal, we enter the rounding routine.
We compute a speculative bound on the denominator as~$M_k := \sqrt{\scalfac_{k+1} /
  (2 \beta^k)}$.
Then the value~$1/(2M_k^2)$ equals~$\beta^k / \scalfac_{k+1} \approx \beta^k
\delta_k$ and tries to estimate the error in the solution.
If reconstruction attempts fail, the term~$\beta^k$ keeps growing exponentially
such that we eventually obtain a true bound on the error.
Initially, however,~$\beta^k$ is small in order to account for the many cases
where the residual $\delta_k$ is a good proxy for the error.
We first apply rational reconstruction to each entry of the primal vector~$x_k$
using denominator bound~$M_k$, denoted by ``\round$((x_k)_i, M_k)$''.
Then we check primal feasibility before proceeding to the dual vector.
If feasibility and optimality could be verified in rational arithmetic, the
rounded solution is returned as optimal.
Otherwise, we compute the next refinement round after which reconstruction
should be tried again.
Because computing continued fraction approximations becomes increasingly
expensive as the encoding length of the approximate solutions grows, rational
reconstruction is executed at a geometric frequency governed by parameter~$f$.
This limits the cumulative effort that is spent on failed reconstruction attempts
at the expense of possibly increasing the number of refinements by a factor of $f$.
The following theorem shows that the algorithm computes an exactly optimal
solution to a primal and dual feasible LP under the conditions guaranteed by
Lemma~\ref{lem:cauchy}.

\begin{algorithm}[t]
\caption{Iterative Refinement with Rational Reconstruction}
\label{algo:lpex:ratrec}
\small
\vspace*{0.5ex}

\Input{rational, primal and dual feasible LP data $A,b,\lb,c$}

\Params{incremental scaling limit $\alpha \in \mathbb{N}, \alpha \geq 2$,
  geometric reconstruction frequency~$f \geq 1$,
  error correction factor~$\beta > 1$}

\Output{optimal primal--dual solution $x^*\in\Q^n,y^*\in\Q^m$}

\Begin{
    initialize Algorithm~\ref{algo:sac} with termination tolerance $\termtol = 0$\;
    $k^*$ \assign $0$\;
    \BlankLine
    \For(\tcc*[f]{refinement loop}){$k \assign 1,2,\ldots$}{
      perform next refinement step in Algorithm~\ref{algo:sac}\;
      $x_k,y_k$ \assign refined numeric solution\;
            $\scalfac_{k+1}$ \assign next scaling factor\;
      \BlankLine

      \uIf(\tcc*[f]{check termination}){$\delta_k = 0$}{
        return $x^* \assign x_k, y^* \assign y_k$\;
      }
      \ElseIf(\tcc*[f]{try reconstruction}){$k \geq k^*$}{
        $M_k$ \assign $\sqrt{\scalfac_{k+1} / 2 \beta^k}$\label{algo:lpex:ratrec:dbound}\;

        \For{$i \assign 1,\ldots,n$\label{algo:lpex:ratrec:forx}}{
          $x^*_i$ \assign \round($(x_k)_i$, $M_k$)\label{algo:lpex:ratrec:x}\;
        }

        \If{$x^*$ exactly feasible in rational arithmetic\label{algo:lpex:ratrec:fory}}{
          \For{$j \assign 1,\ldots,m$}{
            $y^*_j$ \assign \round($(y_k)_j$,$M_k$)\label{algo:lpex:ratrec:y}\;
          }
          \If{$x^*,y^*$ exactly optimal in rational arithmetic}{
            return $x^*,y^*$\;
          }
        }

        \BlankLine
        $k^*$ \assign $\ceil{fk}$\tcc*{counter for next reconstruction}\label{algo:lpex:ratrec:geo}
      }
    }
  }
\end{algorithm}
 
\begin{theorem}\label{lem:lpex:ratrec:conv}   Suppose we are given an LP \eqref{prob:lpeq},
        fixed constants $C \geq 1$, $0 < \epsilon < 1$, $1 < \beta < 1/\epsilon$, and a
  rational limit point~$\tilde{x},\tilde{y}$ with the denominator of each
  component at most~$\tilde{q}$.
  Furthermore, suppose a sequence of primal--dual solutions~$x_k, y_k$ and scaling factors $\scalfac_k \geq 1$ satisfies
  $\maxnorm{(\tilde{x},\tilde{y}) - (x_k,y_k)} \leq C \sum_{i=k+1}^{\infty} \scalfac_i^{-1}$ with $\scalfac_1 = 1$ and $\scalfac_{k} / \scalfac_{k+1} \leq \epsilon$.
  Let $M_k := \sqrt{\scalfac_{k+1} / (2 \beta^k)}$.

  Then there exists $K = K(\tilde{q},C) \in\order(\max\{\size{\tilde{q}},\size{C}\})$ such
  that
  \begin{equation}\label{equ:rrcond}
    \maxnorm{(\tilde{x},\tilde{y}) - (x_k,y_k)} < 1 / (2M_k^2), \; 1 \leq \tilde{q} \leq M_k,
  \end{equation}
  holds for all $k \geq K$, i.e., $\tilde{x},\tilde{y}$ can be reconstructed from
  $x_k,y_k$ componentwise in polynomial time using Theorem~\ref{the:diophantine}.
\end{theorem}

\begin{proof}
    Note that $\scalfac_{k} / \scalfac_{k+1} \leq \epsilon$ for all $k$ implies $\scalfac_{i} \geq \scalfac_{j} \epsilon^{j-i}$ for all $j \leq i$.
    Then $M_k = \sqrt{\scalfac_{k+1} / (2 \beta^k)} \geq \tilde{q}$ holds if $\sqrt{1/(2 \beta^k \epsilon^k)} \geq \tilde{q}$.  This holds for all
    \begin{equation}
    k \geq K_1 := (2\log \tilde{q} + 1) / \log( 1/\beta\epsilon ) \in \order(\size{\tilde{q}}).
  \end{equation}
  Furthermore, $C\sum_{i=k+1}^{\infty} \scalfac_i^{-1} \leq C\sum_{i=k+1}^{\infty} \epsilon^{i-k-1} \scalfac_{k+1}^{-1} = C / ((1-\epsilon) \scalfac_{k+1})$, which is less than $1 / (2M_k^2) = \beta^k / \scalfac_{k+1}$ for all
  \begin{equation}
    k > K_2 := (\log C - \log(1-\epsilon)) / \log\beta \in \order(\size{C}).
  \end{equation}
  Hence \eqref{equ:rrcond} holds for all $k > K := \max\{ K_1, K_2 \}$.
  \qed
\end{proof}

This running time result is output-sensitive as it depends on the encoding length of the
solution.
The value of $C$ is a constant bound on the absolute values in the corrector solutions.
Although $C$ is independent of the input size and does not affect asymptotic
running time, we include it explicitly in order to exhibit the practical
dependency on the corrector solutions returned by the oracle.

We now consider the cost associated with reconstructing the solution vectors.
The cost of applying the standard extended Euclidean algorithm to compute
continued fraction approximations of a rational number with encoding length~$d$ 
is~$O(d^2)$.
Asymptotically faster variants of the extended Euclidean algorithm exist, and can accomplish this goal 
in $O(d \log^2 d \log \log d)$ time~\cite{WangPan2003}, but for simplicity
in our discussion and analysis we use the quadratic bound given above.
The following lemma shows that the total time spent on rational reconstruction
within Algorithm~\ref{algo:lpex:ratrec} is polynomial in the number of refinement rounds
and that if rational reconstruction is applied at a geometric frequency with $f>1$ 
then the total time spent on this task is asymptotically dominated by the
reconstruction of the final solution vector.
We remark that this lemma can be used to show oracle polynomial running time
of the entire algorithm even in the case that $f=1$, but it sheds light on the
fact that choosing $f>1$ can lead to asymptotically improved performance.

\begin{lemma}[Reconstruction of Solution Vectors]\label{lem:lpex:ratrec:reccost}
  The running time of applying rational reconstruction componentwise to $x_k$ and $y_k$
  within the $k$-th round of Algorithm~\ref{algo:lpex:ratrec} is $O((n+m)k^2)$.
  Moreover, if $f>1$ and Algorithm~\ref{algo:lpex:ratrec} terminates at round
  $K$ then the cumulative time spent on rational reconstruction throughout the
  algorithm is $O((n+m){K}^2)$.
\end{lemma}
\begin{proof}
From the proof of Lemma~\ref{lem:polygrowth} we know that at the $k$-th refinement round,
the encoding length of components of $x_k,y_k$ is each bounded
by~$(2\alpha k +3 p + 2)$.
The scaling limit $\alpha$ and the precision level $p$ can be considered as constants
and thus the encoding length of each component is $O(k)$.
Together with the fact that the extended Euclidean algorithm can be implemented to
run in quadratic time in the encoding length of its input, the first result is
established.

To show the second claim we assume that $f>1$ and consider the indices of the 
rounds at which reconstruction is attempted, and use $K$ to denote the
final such index.
We observe that the sequence of these indices, listed in decreasing order, is term-wise
bounded above by the following sequence:
$S = (K, \floor{K /f}, \floor{K /f^2},\ldots, \floor{K /f^a})$
where $a=\log_f{K}$.
This observation follows from the fact that line~\ref{algo:lpex:ratrec:geo}, 
of Algorithm~\ref{algo:lpex:ratrec} involves rounding up to the nearest integer.
Since the encoding length of the components of $x_k,y_k$ increase at each round,
so does the cost of rational reconstruction, so by considering the cost to perform
rational reconstruction at indices in the above sequence $S$ we derive an
upper bound on the true cost.
Now, since the encoding length of each component used for reconstruction is linear
in the iteration index, and the cost for reconstruction is quadratic in this value,
we arrive at the following bound on the total cost, using well known properties of 
geometric series:
\begin{align*}
O\left(\sum_{i=0}^a(n+m)\floor{K/f^i}^2\right) &= O\left((n+m)K^2 \sum_{i=0}^a f^{-2i}\right)
=O\left((n+m)K^2 \sum_{i=0}^\infty f^{-2i}\right)\\
&=O\left((n+m)K^2 \frac{1}{1-f^{-2}}\right)
=O((n+m)K^2)
\end{align*}
This establishes the result. $\hfill\qed$
\end{proof}

Now, assuming the conditions laid out in Theorem~\ref{lem:lpex:ratrec:conv} hold we see that
the number of refinements Algorithm~\ref{algo:lpex:ratrec} performs before computing
an exact rational solution is polynomially bounded in the encoding length of the input.
Together with this bound on the number of refinements, Lemma~\ref{lem:lpex:ratrec:reccost}
gives a polynomial bound on the time spent on rational reconstruction.
Lemma~\ref{lem:polygrowth} and the arguments from Theorem~\ref{the:lpir:polytime} still
apply and limit the growth of the numbers and cost of the other intermediate computations.
Taken together, we arrive at the following.

\begin{theorem}   Suppose we are given a primal and dual feasible LP \eqref{prob:lpeq} and a
  limited-precision LP-basis oracle according to Definition~\ref{def:oracle} with
  constants~$\fpdigits$, $\fptol$, and $\sigma$.
  Fix a scaling limit $\alpha \geq 2$ and let $\epsilon := \max\{\fptol,1/\alpha\}$.
  Then Algorithm~\ref{algo:lpex:ratrec} called with $\beta <
  1/\epsilon$ terminates with an optimal solution in
  oracle-polynomial running time.
\end{theorem}

Note that the basis does not need to be known explicitly.
Accordingly, Algorithm~\ref{algo:lpex:ratrec} may even return an optimal
solution~$x^*,y^*$ that is different from the limit point~$\tilde{x},\tilde{y}$ if
it is discovered by rational reconstruction at an early iterate~$x_k,y_k$.
In this case, $x^*,y^*$ is not guaranteed to be a basic solution unless one
explicitly discards solutions that are not basic during the optimality checks.
 \section{Proofs}
\label{sec:proofs}

This section collects some technical proofs for previous results and the
necessary background material including well-known inequalities regarding
encoding lengths, norms, and systems of equations.

\subsection{Properties of Encoding Lengths}

\begin{lemma}\label{equ:preformulas}
  For~$z_1,\ldots,z_n\in\Z$ and~$r_1,\ldots,r_n\in\Q$,
  \begin{align}
    \size{z_1 + \ldots + z_n} &\leq \size{z_1} + \ldots + \size{z_n},\\
    \size{r_1 \cdot \ldots \cdot r_n} &\leq \size{r_1} + \ldots + \size{r_n},\\
    \size{r_1 + \ldots + r_n} &\leq 2 (\size{r_1} + \ldots + \size{r_n}).\label{lem:intro:sumsize}\\
    \intertext{For matrices~$A\in\Q^{m\times n}$, $B\in\Q^{n\times p}$, $D\in\Q^{n\times n}$,}
    \size{AB} &\leq 2(p\size{A} + m\size{B}),\label{lem:intro:matmulsize}\\
    \size{\det{D}} &\leq 2\size{D} -n^2.\label{lem:intro:detsize}
  \end{align}
\end{lemma}

\begin{proof}
  See \cite[Lemma 1.3.3, 1.3.4, and Exercise~1.3.5]{GroetschelLovaszSchrijver1988}.
  Note that the factor 2 in~\eqref{lem:intro:sumsize} is best possible.
  \qed
\end{proof}

\begin{lemma}\label{lem:intro:normbound}
  For any matrix~$A\in\Q^{m\times n}$,
  \begin{equation}\label{equ:intro:normbound}
    \maxnorm{A} \leq 2^{\size{A} - mn} - mn \leq 2^{\size{A}}.
  \end{equation}
\end{lemma}

\begin{proof}
  \begin{align*}
    \norm{A}_\infty
    &= \max_{i=1,\ldots,m} \sum_{j=1,\ldots,n} \abs{A_{ij}}
    \leq \sum_{i,j} \abs{A_{ij}}
    = (\sum_{i,j} (\abs{A_{ij}} + 1)) - mn\\
    &= 2^{\log(\sum_{i,j} (\abs{A_{ij}} + 1))} - mn
    \leq 2^{\sum_{i,j} \log(\abs{A_{ij}} + 1)} - mn\\
    &= 2^{(\sum_{i,j} \size{A_{ij}}) - mn} - mn
    = 2^{\size{A} - mn} - mn.
  \end{align*}
  \qed
\end{proof}

\begin{lemma}\label{lem:intro:cramer}
  For any nonsingular matrix~$A\in\Q^{n\times n}$, right-hand side~$b\in\Q^n$,
  and solution vector~$x$ of $Ax = b$,
  \begin{align}
    \size{x_i} \leq 4\size{A,b} - 2n^2 \leq 4\size{A,b}.\label{lem:intro:xentry}
    \intertext{for all $i=1,\ldots,n$.  Furthermore,}
    \size{A^{-1}} \leq 4n^2\size{A} - 2n^4 \leq 4n^2\size{A}.\label{lem:intro:inv}
  \end{align}
\end{lemma}

\begin{proof}
  By Cramer's rule, the entries of~$x$ are quotients of determinants of
  submatrices of~$(A,b)$.
  With~\eqref{lem:intro:detsize} this yields \eqref{lem:intro:xentry}.
  For the second inequality note that the $i$-th column of~$A^{-1}$ is the
  solution of~$Ax = e_i$.
  Using Cramer's rule again and the fact that submatrices of $(A,I_n)$ have size
  at most~$\size{A}$ gives inequality~\eqref{lem:intro:inv}.
  \qed
\end{proof}

\subsection{Results from Section~\ref{sec:ir}}
\label{subsec:proofs2}

\noindent
\textbf{Proof of Lemma~\ref{lem:polygrowth}}\quad
  Lines~\ref{algo:sac:scaling1} and~\ref{algo:sac:scaling2} of
  Algorithm~\ref{algo:sac} ensure that $\scalfac_k \leq 2^{\ceil{\log\alpha}
    (k-1)}$ holds at iteration~$k$.
  This can be shown inductively.
  For $k=1$, $\Delta_k = 1 = 2^{\ceil{\log\alpha} (k-1)}$ holds by initialization.
  For $k+1$,
  \begin{align*}
    \Delta_{k+1} = 2^{\ceil{\log(1/\delta_k)} }
    &\leq 2^{\ceil{\log(\alpha\scalfac_k)}}
    \leq 2^{\ceil{\log\alpha + \log\scalfac_k}}
    \leq 2^{\ceil{\log\alpha + \ceil{\log\alpha} (k-1)}}\\
    &\leq 2^{\ceil{\ceil{\log\alpha} k}}
    = 2^{\ceil{\log\alpha} k}.
  \end{align*}
  As a result, $\size{\scalfac_k} \leq \ceil{\log\alpha} k$.

  Furthermore, by induction from line~\ref{algo:sac:correct}, the entries in the refined solution vectors
  $x_k$ and $y_k$ have the form
  \begin{align}
    \sum_{j = 1}^k {\scalfac_j}^{-1} \frac{n_j}{2^\fpdigits}
  \end{align}
  with~$n_j \in\Z$, $\abs{n_j} \leq 2^{2\fpdigits}$, for~$j =
  1,\ldots,k$.
  \newcommand{\logalpha}{a}
  With $\scalpow_j := \log(\scalfac_j)$ and $\logalpha := \ceil{\log(\alpha)}$ this can be rewritten as
  \begin{align}
    \sum_{j = 1}^k 2^{-\scalpow_j} \frac{n_j}{2^\fpdigits} =
    \big( \sum_{j = 1}^k n_j 2^{\logalpha(k-1) - \scalpow_j} \big) / 2^{\fpdigits + \logalpha(k-1)}.
  \end{align}
  The latter is a fraction with integer numerator and denominator.  The numerator is bounded as follows:
  \begin{equation}
  \begin{aligned}
    \bigabs{ \sum_{j = 1}^k n_j 2^{\logalpha(k-1) - \scalpow_j} }
    &\leq \sum_{j = 1}^k \abs{n_j} 2^{\logalpha(k-1) - \scalpow_j}
    \leq 2^{2\fpdigits} \sum_{i=0}^{\logalpha(k-1)} 2^i\\
    &\leq 2^{2\fpdigits} ( 2^{\logalpha(k-1) + 1} - 1 )
    \leq 2^{2\fpdigits + \logalpha k} - 1.
  \end{aligned}
  \end{equation}
  Hence, the size of the entries of $x_k$ and $y_k$ grows only linearly with the number of iterations~$k$,
  \begin{equation}
  \begin{aligned}
    \size{x_k} + \size{y_k}
    &\leq (n+m) \Big(\size{\sum_{j = 1}^k n_j 2^{\logalpha(k-1) - \scalpow_j}} + \size{2^{\fpdigits + \logalpha(k-1)}}\Big)\\
    &\leq (n+m) \Big( 2 + \ceil{\log(2^{2\fpdigits + \logalpha k})} + \ceil{\log(2^{\fpdigits + \logalpha(k-1)} + 1)} \Big)\\
    &\leq (n+m) (2 \logalpha k + 3\fpdigits + 2).
  \end{aligned}
  \end{equation}
  The size of the remaining numbers set at iteration $k$ are bounded accordingly,
  \begin{align}
    \size{\hat b} &\leq 4(\size{b} + \size{A} + \size{x_k}),\\
    \size{\hat \lb} &\leq 2(\size{\lb} + \size{x_k}),\\
    \size{\hat c} &\leq 4(\size{c} + \size{A} + \size{y_k}),\\
    \size{\delta_k} &\leq \max\{ \size{\hat b}, \size{\hat \lb}, \size{\hat c}, 2(\size{\hat\lb} + \size{\hat c}), \size{\alpha}\size{\scalfac_k} \}.
  \end{align}
  By Lemma~\ref{lem:convrate}, the maximum number of iterations
  is~$\order(\log(1/\termtol)) = \order(\size{\termtol})$.
  To summarize, the encoding length of any of the numbers encountered during the course
  of the algorithm is $\order(\size{A,b,\lb,c}+ (n+m)\size{\termtol})$.
  \qed

\subsection{Results from Section~\ref{sec:oracle}}
\label{subsec:proofs3}

\noindent
\textbf{Proof of Lemma~\ref{lem:lpex:basisviolation}}\quad
  Let $x,y$ be a basic primal--dual solution with respect to\ some basis~$\B$.
  Let $\N = \{1,\ldots,n\} \without \B$, let~$B = A_{\cdot\B}$ be the
  corresponding (square, non-singular) basis matrix and $N = A_{\cdot\N}$ the
  matrix formed by the nonbasic columns.
  Then the primal solution is given as solution of
  \begin{align}
    \underbrace{\left(\begin{matrix}
      B & N \\ 0 & \unitmatrix_{n - m}
    \end{matrix}\right)}_{\enifed \tilde B \in \Q^{n\times n}}
    \left(\begin{matrix}
      x_\B \\ x_\N
    \end{matrix}\right)
    &= \underbrace{\left(\begin{matrix}
      b \\ \lb_\N
    \end{matrix}\right)}_{\enifed \tilde b \in\Q^n}.\label{equ:lpex:primalrowbasis}
    \intertext{The dual vector~$y$ is determined by}
    \tilde B^\T \left(\begin{matrix}
      y \\ z
    \end{matrix}\right)
    &= c\label{equ:lpex:dualrowbasis}
  \end{align}
  together with the vector $z\in\Q^{n-m}$ containing the dual slacks of the
  nonbasic variables.

  First, primal infeasibilities can only stem from violations of the bound
  constraints on basic variables since the equality constraints and the nonbasic
  bounds are satisfied by construction.
  Hence, they are of the form~$\abs{x_i - \lb_i}$ for some $i\in\B$.
  From Lemma~\ref{lem:intro:cramer} applied to~\eqref{equ:lpex:primalrowbasis}
  we know that the entries in $x_\B$ have size at most~$4\size{\tilde B,\tilde b} -
  2n^2$.
  Because
  \begin{align*}
    \size{\tilde B, \tilde b} &= \size{B} + \size{N} + \size{0} + \size{I_{n-m}} + \size{b}+ \size{\lb_\N}\\
    &= \size{A,b,\lb_\N} + (n + 1)(n - m)\\
    &\leq \size{A,b,\lb} + (n + 1)(n - m),
  \end{align*}
  it follows that all nonzero entries of $x_\B$ are of form $p/q$,
  $p\in\Z,q\in\nonneg\Z$, with
  \begin{align*}
    q \leq 2^{4\size{A,b,\lb} + 4(n + 1)(n - m) - 2n^2} \leq 2^{4\size{A,b,\lb} + 2n^2 + 4n}.
  \end{align*}
  The entries in $\lb$ can be written as $p'/q'$, $p'\in\Z,q'\in\nonneg\Z$, with
  $q' \leq 2^{\size{\lb}}$.
  Combining this we know for all nonzero primal violations
  \begin{align*}
    \abs{x_i - \lb_i}
    &= \bigabs{\frac{p}{q} - \frac{p'}{q'}}
    = \frac{\abs{pq' - p'q}}{qq'}\\
    &\geq 1 / (2^{4\size{A,b,\lb} + 2n^2 + 4n} \cdot 2^{\size{\lb}})
    = 1 / 2^{4\size{A,b} + 5\size{\lb} + 2n^2 + 4n}.
  \end{align*}

  Second, note that dual infeasibilities are precisely the (absolute values of
  the) negative entries of $z$ in~\eqref{equ:lpex:dualrowbasis}.
  Again from Lemma~\ref{lem:intro:cramer} we have
  \begin{align*}
    \size{z_j}
    &\leq 4\size{\tilde B, c} - 2n^2\\
    &\leq 4\size{A,c} + 4(n + 1)(n - m) - 2n^2\\
    &\leq 4\size{A,c} + 2n^2 + 4n
  \end{align*}
  for all~$j \in \N$, and so any nonzero dual violation is at least~$1 /
  2^{4\size{A, c} + 2n^2 + 4n}$.
  \qed
\bigskip

\noindent
\textbf{Proof of Theorem~\ref{the:lpex:optconv}}\quad
  The derivations presented in the following all refer to objects at one fixed
  iteration~$k$.
  Hence, the iteration index is dropped for better readability, i.e., we write
  $x$ for $x_k$ and $y$ for $y_k$, and $\B$ for $\B_k$.
  Furthermore, define $\tilde x, \tilde y$ to be the exact basic solution vectors
  corresponding to current basis $\B$, let~$\N = \{1,\ldots,n\} \without \B$, and
  let~$B = A_{\cdot\B}$ and~$N = A_{\cdot\N}$.
  We will show that for a sufficiently large $k$, the primal and dual violations
  of the exact basic solution vectors drop below the minimum infeasibility
  thresholds from Lemma~\ref{lem:lpex:basisviolation}.

  By construction, the basic solution~$\tilde x$ satisfies the equality
  constraints $Ax = b$ exactly.
  For the violation of the lower bounds, we first show that~$x$
  and~$\tilde x$ converge towards each other.

  \begin{claim}{Claim 1}
    At iteration $k = 1,2,\ldots$, $\maxnorm{x - \tilde x} \leq
    2^{4m^2\size{A} + 1} \epsilon^k$.
  \end{claim}

  \noindent
  For the nonbasic variables we have
  \begin{align*}
    \maxnorm{x_{\N} - \tilde x_{\N}}
    = \maxnorm{x_{\N} - \lb_{\N}}
    \stackrel{\eqref{point4}}{\leq} \epsilon^k.
  \end{align*}
  For the basic variables we have
 \begin{align*}
   \maxnorm{x_{\B} - \tilde x_{\B}}
   &= \maxnorm{B^{-1} (B x_{\B} - B \tilde x_{\B})}\\
   &\leq \maxnorm{B^{-1}} \maxnorm{B x_{\B} - B \tilde x_{\B}}\\
   &= \maxnorm{B^{-1}} \maxnorm{\underbrace{B x_{\B}}_{\clap{\scriptsize $= Ax - N x_\N$}} - (b - N \lb_{\N})}\\
   &= \maxnorm{B^{-1}} \maxnorm{A x - b - (Nx_\N - N\lb_{\N})}\\
   &\leq \maxnorm{B^{-1}} \big( \maxnorm{A x - b} + \maxnorm{N (x_{\N} - \lb_{\N})} \big)\\
   &\leq \maxnorm{B^{-1}} \big( \underbrace{\maxnorm{A x - b}}_{\leq \epsilon^k \text{ by \eqref{point1}}} + \maxnorm{N} \underbrace{\maxnorm{x_{\N} - \lb_{\N}}}_{\leq \epsilon^k \text{ by \eqref{point4}}} \big)\\
   &\leq \maxnorm{B^{-1}} \big( \maxnorm{N} + 1 \big) \epsilon^k.
 \end{align*}
  By Lemma~\ref{lem:intro:normbound} and Lemma~\ref{lem:intro:cramer},
  \begin{align*}
    \maxnorm{B^{-1}} & \leq 2^{\size{B^{-1}}} \leq 2^{4m^2\size{B}},    \intertext{and}
    \maxnorm{N} + 1 &\leq 2^{\size{N}} + 1 \leq 2^{\size{N} + 1},
    \intertext{hence}
    \maxnorm{B^{-1}} \big( \maxnorm{N} + 1 \big) &\leq 2^{4m^2\size{B}} \cdot 2^{\size{N} + 1}
    \leq 2^{4m^2\size{A} + 1},
  \end{align*}
  proving the claim.
    
  \pagebreak

  \begin{claim}{Claim 2}
    At iteration~$k = 1,2,\ldots$, $\tilde x_i - \lb_i \geq
    -2^{4m^2\size{A} + 2} \epsilon^k$ for all $i\in\{1,\ldots,n\}$.
  \end{claim}

  \noindent
  This follows from
  \begin{align*}
    \tilde x_i - \lb_i
    &= \tilde x_i - x_i + x_i - \lb_i\\
    &\geq -\maxnorm{\tilde x_i - x_i} + x_i - \lb_i\\
    &\geq -2^{4m^2\size{A} + 1} \epsilon^k - \epsilon^k
    \geq -2^{4m^2\size{A} + 2} \epsilon^k.
  \end{align*}

  For dual feasibility, we first show that the dual solutions~$y$
  and~$\tilde y$ converge towards each other.

  \begin{claim}{Claim 3}
    At iteration $k = 1,2,\ldots$, $\maxnorm{y - \tilde y} \leq 2^{4m^2\size{B}}
    \epsilon^k$.
  \end{claim}

  \noindent
  This follows from
  \begin{align*}
    \maxnorm{y - \tilde y}
    &= \maxnorm{(B^\T)^{-1} B^\T (y - \tilde y)}
    \leq \maxnorm{(B^\T)^{-1}} \maxnorm{B^\T (y - \tilde y)}\\
    &= \maxnorm{(B^\T)^{-1}} \underbrace{\maxnorm{B^\T y - c_{\B}}}_{= \max\{ \abs{c_i - y^\T A_{\cdot i}} : i \in \B \}}
    \stackrel{\eqref{point5}}{\leq} \maxnorm{(B^\T)^{-1}} \,\epsilon^k
  \end{align*}
  and
 \begin{align*}
   \maxnorm{(B^\T)^{-1}}
   &\leq 2^{\size{(B^\T)^{-1}}} \leq 2^{4m^2\size{B^\T}} \leq 2^{4m^2\size{B}}
 \end{align*}
  using Lemma~\ref{lem:intro:normbound} and Lemma~\ref{lem:intro:cramer}.

  \begin{claim}{Claim 4}
    At iteration $k = 1,2,\ldots,$, $c_i - \tilde y^\T A_{\cdot i} \geq -2^{4m^2\size{A} + 1} \epsilon^k$ for all $i\in\{1,\ldots,n\}$.
  \end{claim}

  \noindent
  For the basic variables, $c_i - \tilde y^\T A_{\cdot i} = 0$.   For~$i\in\N$,
  \begin{align*}
    \abs{(y - \tilde y)^\T A_{\cdot i}}
    &\leq \maxnorm{A_{\cdot i}} \maxnorm{y - \tilde y}\\
    &\leq \maxnorm{N}\, 2^{4m^2\size{B}} \epsilon^k\\
    &\leq 2^{\size{N}} 2^{4m^2\size{B}} \epsilon^k\\
    &\leq 2^{4m^2\size{B} + \size{N}} \epsilon^k\\
    &= 2^{4m^2\size{A}} \epsilon^k.
  \end{align*}
  This proves the claim via
  \begin{align*}
    c_i - \tilde y^\T A_{\cdot i}
    &= c_i - y^\T A_{\cdot i} + (y - \tilde y)^\T A_{\cdot i}\\
    &\geq c_i - y^\T A_{\cdot i} - \abs{(y - \tilde y)^\T A_{\cdot i}}\\
    &\geq -\epsilon^k - 2^{4m^2\size{A}} \epsilon^k\\
    &\geq -2^{4m^2\size{A} + 1} \epsilon^k.
  \end{align*}

  From Claim~2 and 4, $\tilde x$ and $\tilde y$ violate primal and dual
  feasibility by at most~$2^{4m^2\size{A} + 2} \epsilon^k$ and~$2^{4m^2\size{A}
    + 1} \epsilon^k$, respectively.
  These values drop below the thresholds from
  Theorem~\ref{lem:lpex:basisviolation} as soon as
  \begin{align*}
   2^{4m^2\size{A} + 2} \epsilon^k &< 1 / 2^{4\size{A,b} + 5\size{\lb} + 2n^2 + 4n}\\
   \Leftrightarrow \epsilon^k &< 1 / 2^{4m^2\size{A} + 2 + 4\size{A,b} + 5\size{\lb} + 2n^2 + 4n}\\
   \Leftrightarrow k &> \frac{4m^2\size{A} + 4\size{A,b} + 5\size{\lb} + 2n^2 + 4n + 2}{\log(1/\epsilon)} \enifed K_P
   \intertext{and}
   2^{4m^2\size{A} + 1} \epsilon^k &< 1 / 2^{4\size{A, c} + 2n^2 + 4n}\\
   \Leftrightarrow \epsilon^k &< 1 / 2^{4m^2\size{A} + 1 + 4\size{A, c} + 2n^2 + 4n}\\
   \Leftrightarrow k &> \frac{4m^2\size{A} + 4\size{A, c} + 2n^2 + 4n + 1}{\log(1/\epsilon)} \enifed K_D.
  \end{align*}
  From Theorem~\ref{lem:lpex:basisviolation}, the solution $\tilde x,\tilde
  y$ must be primal and dual feasible for $k \geq K := \max\{ K_P, K_D\} + 1$.
  Since the solution is basic, $\tilde x$ and $\tilde y$ also satisfy
  complementary slackness, and hence they are optimal.
  The resulting threshold $K$ has the order claimed.
  \qed
 \section{Computational Experiments}
\label{sec:lpex:experiments}

Despite the theoretical analysis it remains an open question which of the proposed methods
for solving linear programs exactly will perform best empirically, iterative
refinement with basis verification or with rational reconstruction.
We implemented both algorithms in the simplex-based LP~solver
\soplex~\cite{SoPlexWeb} in order to analyze and compare their computational
performance on a large collection of LP instances from publicly available test
sets.
In particular, we aim to answer the following questions:

\begin{itemize}
\item How many instances can be solved by each algorithm and how do they compare in
running time?
\item How much of the solving time is consumed by rational factorization and rational
reconstruction, and how often are these routines called?
\item How many refinements are needed until reconstruction succeeds, and are the
reconstructed solutions typically basic?
\end{itemize}
In addition, we compared both methods against the state-of-the-art solver
\qsoptex, which is based on incremental precision
boosting~\cite{QsoptexWeb,ApplegateCookDashEspinoza2007}.

\subsection{Implementation and Experimental Setup}

Both methods---basis verification and rational reconstruction---were implemented as
extensions to the existing LP iterative refinement procedure of SoPlex, which is
detailed in~\cite{GleixnerSteffyWolter2016}.
In the following, these extensions are denoted by~\soplexfac and~\soplexrec,
respectively.

The exact solution of the primal and dual basis systems relies on a rational
LU~factorization and two triangular solves for the standard, column-wise basis
matrix containing slack columns for basic rows.
The implementation of the rational solves is an adjusted version of the
floating-point LU code of SoPlex, removing parts specific to floating-point
operations.
The stalling threshold~$L$ for calling rational factorization is set to $2$,
i.e., factorization will only be called after two refinement steps have not
updated the basis information.

The rational reconstruction routine is an adaptation of the code used
in~\cite{Steffy2011}.
The newly introduced error correction factor~$\beta$ was set to~1.1.
Furthermore, since nonbasic variables are held fixed at one of their bounds, the
corresponding entries of the primal vector can be skipped during reconstruction.
The rational reconstruction frequency~$f$ is set to $1.2$, i.e., after a
failed attempt at reconstructing an optimal solution, reconstruction is paused
until 20\% more refinement steps have been performed.
We also employ the DLCM method described in \cite{CookSteffy2011} (see 
also~\cite{ChenStorjohann2005}) for accelerating the reconstruction of the primal and dual
solution vectors.

As test bed we use a set of 1,202~primal and dual feasible LPs collected from
several publicly available sources: \netlib~\cite{netlib}, Hans Mittelmann's
benchmark instances~\cite{MittelmannLp}, Csaba M\'esz\'aros's LP
collection~\cite{MeszarosLp}, and the LP relaxations of the \coral and \miplib
mixed-integer linear programming libraries~\cite{CoralWeb,MiplibWeb}.  For
details regarding the compilation we refer to the electronic supplement
of~\cite{GleixnerSteffyWolter2016}.

The experiments were conducted on a cluster of 64-bit Intel Xeon X5672~CPUs at
3.2\,GHz with 48\,GB main memory, simultaneously running at most one job per node.
\soplex was compiled with \gcc~4.8.2 and linked to the external libraries
\gmp~5.1.3 and \eglib~2.6.20.
\qsoptex was run in version~2.5.10.
A time limit of two hours per instance was set for each \soplex and \qsoptex
run.

\subsection{Results}

For the evaluation of the computational experiments we collected the following statistics:
the total number of simplex iterations and running
times for each solver; for \qsoptex the maximum floating-point precision used;
for the \soplex runs the number of refinement steps and the number and execution
times for basis verification and rational reconstruction, respectively.
For the solutions returned by \soplexfac and \soplexrec, we additionally
computed the least common multiple of the denominators of their nonzero entries
and report their order of magnitude, as an indicator for how ``complicated'' the
representation of the exact solution is.
The electronic supplement provides these statistics for each instance of the
test set.
Tables~\ref{tab:lpex:vsqsoptex} and~\ref{tab:lpex:facrec-aggr} below report an
aggregated summary of these results.
In the following, we discuss the main observations.

\paragraph{Overall Comparison}

None of the solvers dominates the others meaning that for each of \qsoptex,
\soplexfac, and \soplexrec there exist instances that can be solved only by this
one solver.
Overall, however, the iterative refinement-based methods are able to solve more
instances than \qsoptex, and \soplexfac exhibits significantly shorter running
times than the other two methods.
Of the 1,202~instances, 1,158 are solved by all three within the available time
and memory resources.
\qsoptex solves 1,163~instances, \soplexrec solves 1,189~instances, and \soplexfac solves the largest number of instances: 1,191.
Regarding running time, \qsoptex is fastest 324~times, \soplexrec 569~times, and
\soplexfac is fastest for 702~instances.\footnote{In order to account fairly for
  small arbitrary deviations in time measurements, an algorithm is considered
  ``fastest'' if it solves an instance in no more than 5\% of the time taken by
  the fastest method.  Hence, the numbers do not add up to the total of
  1,202~instances. Furthermore, note that the same picture holds also when
  excluding easy instances that took less than two seconds by all solvers:
  \soplexfac wins more than twice as often as \qsoptex.}

For 32 of the 44~instances not solved by \qsoptex, this is due to the time limit.
On the other 12~instances, it cannot allocate enough memory on the 48\,GB
machine.
Eight times this occurs during or after precision boosts and points to the
disadvantage that keeping and solving extended-precision LPs may not only be
time-consuming, but also require excessive memory.
By contrast, the iterative refinement-based methods work with a more
memory-efficient double-precision floating-point rounding of the LP and never
reach the memory limit.
However, \soplexrec and \soplexfac could not solve seven instances because of
insufficient performance of the underlying floating-point oracle.\footnote{  On three instances,
    floating-point \soplex could not solve the first refinement LP within the time limit.
    For three further instances, the initial floating-point solve incorrectly claimed
  unboundedness and for one instance it incorrectly claimed infeasibility;
      in all of these cases, the incorrect claims were rejected successfully using feasibility and
  unboundedness tests as described in~\cite{GleixnerSteffyWolter2016}, but after
  starting to refine the original LP again, floating-point \soplex failed to
  return an approximately optimal solution even when trying to run with different floating-point settings.
}

Finally, for the 492~instances that could be solved by all three algorithms, but
were sufficiently nontrivial such that one of the solvers took at least two
seconds, Table~\ref{tab:lpex:vsqsoptex} compares average running times and
number of simplex iterations.
In addition to all 492~instances, the lines starting with 64-bit, 128-bit, and
192-bit filter for the subsets of instances corresponding to the final precision
level used by \qsoptex.
It can be seen that \soplexfac outperforms the other two algorithms.
Overall, it is a factor of 1.85~faster than \qsoptex and even 2.85~times faster
than \soplexrec.
On the instances where \qsoptex found the optimal solution after the
double-precision solve (line 64-bit), \soplexfac is 30\% faster
although it uses about 40\% more simplex iterations than \qsoptex.
Not surprisingly, when \qsoptex has to boost the working precision of the
floating-point solver (lines 128-bit and 192-bit), the results become even more
pronounced, with \soplexfac being over three times faster than \qsoptex.

\begin{table}
  \caption{Aggregate comparison of solvers \qsoptex and exact \soplex with basis
    verification (\soplexfac) and rational reconstruction (\soplexrec) on
    instances that could be solved by all and where one solver took at least
    2~seconds.
    Columns~\itercol\xspace and \timecol\xspace report shifted geometric means of
    simplex iterations and solving times, using a shift of 2~seconds and
    100~simplex iterations, respectively.
    Column~\speedupcol\xspace reports the ratio between the mean solve times of \soplex and~\qsoptex.}
    \label{tab:lpex:vsqsoptex}
  \setlength{\tabcolsep}{1pt}
  \medskip
  
  \scriptsize
  \begin{tabular*}{\textwidth}{@{\extracolsep{\fill}}lrrrrrrrrr}
                        \toprule
    &
    & \multicolumn{2}{c}{\qsoptex}
    & \multicolumn{3}{c}{\soplexfac}
    & \multicolumn{3}{c}{\soplexrec}\\
    \cmidrule(l){3-4}\cmidrule(l){5-7}\cmidrule(l){8-10}
    \qsoptex \preccol & \numinstcol
    & \itercol & \timecol
    & \itercol & \timecol & \speedupcol
    & \itercol & \timecol & \speedupcol\\
    \midrule

  any &   492 &     8025.7 &    15.6 &     9740.6 &     8.5 &    0.54 &     9740.6 &    24.2 &    1.55\\
   64-bit &   324 &     8368.3 &    16.1 &    11683.7 &    11.3 &     0.70 &    11683.7 &    14.8 &    0.92\\
  128-bit &   163 &     7217.1 &    13.9 &     6757.2 &     4.3 &    0.31 &     6757.2 &    58.5 &    4.21\\
  192-bit &     5 &    16950.9 &    72.5 &    10763.4 &    20.5 &    0.28 &    10763.4 &   134.6 &    1.86\\
      \bottomrule
  \end{tabular*}
\end{table}
 
The remaining analysis looks at the results of the iterative refinement-based
methods in more detail.
Table~\ref{tab:lpex:facrec-aggr} provides a summary of the statistics in the
electronic supplement for all 1,186~instances that are solved by both \soplexrec
and \soplexfac.
The lines starting with \timebracketgeq{t} filter for the subsets of
increasingly hard instances for which at least one method took $t=1,10,100$
seconds.

\paragraph{Rational Reconstruction}

The results largely confirm the predictions of
Theorem~\ref{the:diophantine},
which expects an approximate solution with error about~$10^{-2\,\text{\dlcmcol}}$ or less.
Here \dlcmcol~is the $\log_{10}$ of the least common multiple of the
denominators in the solution vector as reported in the electronic supplement.
Indeed, the \dlcmcol~value mostly correlates with the number of refinement
rounds, though several instances exist where reconstruction succeeds with even
fewer refinements than predicted.
As can be seen from column ``\rectimecol'', the strategy of calling rational
reconstruction at a geometric frequency succeeds in keeping reconstruction time
low also as the number of refinements increases.

The 5~instances that could be solved by \soplexfac, but not by \soplexrec, show
large \dlcmcol~value.
This helps to explain the time outs and points to a potential bottleneck of
\soplexrec.
The number of refinements that could be performed within the time limit simply
did not suffice to produce an approximate solution of sufficiently high
accuracy.
Finally, the reconstructed solution was almost always basic and showed identical
\dlcmcol~value as the solution of \soplexfac.
For 7~instances, rational reconstruction computed a non-basic solution.

\begin{table}
  \caption{Computational results for iterative refinement with basis
    verification (\soplexfac) and rational reconstruction (\soplexrec).
    Columns~\irrefcol, \faccol, \reccol\xspace contain arithmetic means of the number of
    refinements, basis verifications, and reconstruction attempts, respectively.
    Columns~\timecol, \factimecol, \rectimecol\xspace report shifted geometric
    mean times for the total solving process, the basis verifications, and
    rational reconstruction routines, respectively, with a shift of 2~seconds.
    Column~\speedupcol\xspace reports the ratio between the mean solve times of \soplexrec and~\soplexfac.}
  \label{tab:lpex:facrec-aggr}
  \setlength{\tabcolsep}{1pt}
  \medskip

    \scriptsize
  \begin{tabular*}{\textwidth}{@{\extracolsep{\fill}}lrrrrrrrrrr}
                                    \toprule
    &
    & \multicolumn{4}{c}{\soplexfac}
    & \multicolumn{5}{c}{\soplexrec}\\
    \cmidrule(l){3-6}\cmidrule(l){7-10}
    Test set & \numinstcol
    & \hspace*{1.5em}\irrefcol & \faccol & \factimecol & \timecol
    & \hspace*{2.3em}\irrefcol & \reccol & \rectimecol & \timecol
    & \speedupcol\\
    \midrule

all                            &  1186 &      2.1 &     0.95 &     0.21 &      2.8 &     68.3 &     6.74 &     1.26 &      5.2 &     1.82\\
\timebracketgeq{1}             &   591 &      2.3 &     0.98 &     0.43 &      8.8 &    135.1 &    11.04 &     3.28 &     21.8 &     2.47\\
\timebracketgeq{10}            &   311 &      2.4 &     0.99 &     0.83 &     24.1 &    241.6 &    15.47 &     9.15 &    101.9 &     4.22\\
\timebracketgeq{100}           &   161 &      2.7 &     0.98 &     1.40 &     42.8 &    384.9 &    19.70 &    22.95 &    340.3 &     7.95\\
             \bottomrule
  \end{tabular*}
  \end{table}

\paragraph{Basis Verification}

Compared to \soplexrec, the number of refinements for \soplexfac is very small,
because the final, optimal basis is almost always reached by the second round.
This confirms earlier results of~\cite{GleixnerSteffyWolter2016}.
Accordingly, for most LPs, \soplexfac performs exactly one rational factorization
(1,123 out of 1,191 solved); for 7~instances two factorizations.

Notably, there are 61~instances where no factorization is necessary because the
approximate solution is exactly optimal.
This is explained by the fact that the numbers in the solution have small
denominator.
For 59~instances, the denominator is even one, i.e., the solution is integral.
As a result, the average number of factorizations (column ``\faccol'') of
Table~\ref{tab:lpex:facrec-aggr} is slightly below one.
This situation even occurs for LPs with longer running times, since the
simplicity of the solution is not necessarily correlated with short running
times of the floating-point simplex.

On average, the time for rational factorization and triangular solves (column
``\factimecol'') is small compared to the total solving time.
Also in absolute values, $\factimecol$ is small for the vast majority of
instances: for 901~instances it is below 0.1~seconds.
In combination with the small number of refinements needed to reach the optimal
basis, this helps to explain why \soplexfac is on average between 1.82~and
7.95~times faster than \soplexrec.
However, for 21~instances, $\factimecol$ exceeds 7~seconds and consumes more than
90\% of the running time.
On 3~of these instances, \soplexfac times out, while they can be solved by
\soplexrec.

 \section{Conclusion}
\label{sec:lpex:conclusion}

This paper developed and analyzed two new algorithms for exact linear programming
over the rational numbers.
A notion of limited-precision LP oracles was formalized, which closely resembles
modern floating-point simplex implementations.
The methods extend the iterative refinement scheme
of~\cite{GleixnerSteffyWolter2016} in conceptually different directions: basis
verification using rational linear systems solves and rational reconstruction
using the extended Euclidean algorithm.
Both are proven to converge to an optimal basic solution in oracle-polynomial time.

Computational experiments revealed that the rational factorization approach solved
slightly more instances within a time limit and was about 46\% faster on average.
However, several instances were identified that were solved much faster by rational
reconstruction; we also found that the reconstruction approach was slightly faster
for those LPs with very short running times.
This raises the question how to combine both techniques most efficiently into a
hybrid algorithm.
An immediate idea would be to perform a rational factorization only after rational
reconstruction has failed for a fixed number of refinements.
However, the critical instances on which reconstruction wins typically have
solutions with large denominators and require a high number of refinements.
Putting the factorization on hold in the meantime would incur a major slowdown
on the majority of instances.
Hence, currently the most promising hybridization seems to be a straightforward
parallelization:
whenever iterative refinement reaches a basis candidate that is assumed to be
optimal, a rational factorization can be performed in the background while
refinement and reconstruction is continued in the foreground.

Finally, we compared the iterative refinement based algorithms against the
current state-of-the-art approach, the incremental precision boosting procedure
implemented by the solver \qsoptex.
We found that \soplex (using the rational factorization strategy) is 1.85 to 3
times faster on our test set and solves more instances within given time and
memory restrictions.
However, the advantage of incremental precision boosting is its capability to
handle extremely ill-conditioned LPs by increasing the working precision of the
floating-point solver when necessary.
In order to harness the strengths of both approaches, incremental precision
boosting can be integrated into iterative refinement quite naturally:
whenever the underlying floating-point solver encounters numerical difficulties and
fails to return a satisfactory approximate solution, boost the precision of the
floating-point solver to the next level.
This would help to handle instances on which iterative refinement failed in our
experiments because \soplex's double-precision simplex broke down,
while for the vast majority of instances no precision boosts would
be necessary, retaining the significant performance benefits of iterative
refinement.
 \bigskip
\bigskip

\noindent\textbf{Acknowledgements}\medskip

\noindent
The authors would like to thank the anonymous reviewers for their detailed study
of the manuscript and their comments, which were of exceptionally high
quality.
\bigskip
\bigskip

\noindent\textbf{Online supplement}\medskip

\noindent
The online supplement is available with the \emph{Mathematical Programming}
publication at \url{http://dx.doi.org/10.1007/s10107-019-01444-6}.
Please cite this work always via the original publication at the above link.
\bigskip

\renewcommand{\refname}{\normalsize References}
\setlength{\bibsep}{0.25ex plus 0.3ex}
\bibliographystyle{abbrvnat}

\begin{small}
\bibliography{exlpir}

\begin{thebibliography}{32}
\providecommand{\natexlab}[1]{#1}
\providecommand{\url}[1]{\texttt{#1}}
\expandafter\ifx\csname urlstyle\endcsname\relax
  \providecommand{\doi}[1]{doi: #1}\else
  \providecommand{\doi}{doi: \begingroup \urlstyle{rm}\Url}\fi

\bibitem[Abbott and Mulders(2001)]{AbbottMulders2001}
J.~Abbott and T.~Mulders.
\newblock How tight is {H}adamard's bound?
\newblock \emph{Experimental Mathematics}, 10\penalty0 (3):\penalty0 331--336,
  2001.
\newblock URL \url{http://projecteuclid.org/euclid.em/1069786341}.

\bibitem[Applegate et~al.()Applegate, Cook, Dash, and Espinoza]{QsoptexWeb}
D.~L. Applegate, W.~Cook, S.~Dash, and D.~G. Espinoza.
\newblock {Qsopt\_ex}.
\newblock URL \url{http://www.dii.uchile.cl/~daespino/ESolver_doc/}.

\bibitem[Applegate et~al.(2007)Applegate, Cook, Dash, and
  Espinoza]{ApplegateCookDashEspinoza2007}
D.~L. Applegate, W.~Cook, S.~Dash, and D.~G. Espinoza.
\newblock Exact solutions to linear programming problems.
\newblock \emph{Operations Research Letters}, 35\penalty0 (6):\penalty0
  693--699, 2007.
\newblock \doi{10.1016/j.orl.2006.12.010}.

\bibitem[Azulay and Pique(1998)]{AzulayPique1998}
D.-O. Azulay and J.-F. Pique.
\newblock Optimized {$Q$}-pivot for exact linear solvers.
\newblock In M.~Maher and J.-F. Puget, editors, \emph{Principles and Practice
  of Constraint Programming — CP98}, volume 1520 of \emph{Lecture Notes in
  Computer Science}, pages 55--71. Springer, 1998.
\newblock \doi{10.1007/3-540-49481-2_6}.

\bibitem[Chen and Storjohann(2005)]{ChenStorjohann2005}
Z.~Chen and A.~Storjohann.
\newblock A {BLAS} based {C} library for exact linear algebra on integer
  matrices.
\newblock In \emph{Proceedings of the 2005 International Symposium on Symbolic
  and Algebraic Computation}, ISSAC '05, pages 92--99, 2005.
\newblock \doi{10.1145/1073884.1073899}.

\bibitem[Cheung and Cucker(2006)]{CheungCucker2006}
D.~Cheung and F.~Cucker.
\newblock Solving linear programs with finite precision: {II}. algorithms.
\newblock \emph{Journal of Complexity}, 22\penalty0 (3):\penalty0 305 -- 335,
  2006.
\newblock ISSN 0885-064X.
\newblock \doi{https://doi.org/10.1016/j.jco.2005.10.001}.
\newblock URL
  \url{http://www.sciencedirect.com/science/article/pii/S0885064X05000956}.

\bibitem[{Computational Optimization Research At Lehigh}()]{CoralWeb}
{Computational Optimization Research At Lehigh}.
\newblock {MIP} instances.
\newblock URL
  \url{http://coral.ie.lehigh.edu/data-sets/mixed-integer-instances/}.

\bibitem[Cook and Steffy(2011)]{CookSteffy2011}
W.~Cook and D.~E. Steffy.
\newblock Solving very sparse rational systems of equations.
\newblock \emph{{ACM} Transactions on Mathematical Software}, 37\penalty0
  (4):\penalty0 39:1--39:21, 2011.
\newblock \doi{10.1145/1916461.1916463}.

\bibitem[Dhiflaoui et~al.(2003)Dhiflaoui, Funke, Kwappik, Mehlhorn, Seel,
  Sch\"{o}mer, Schulte, and Weber]{DhiflaouiEtAl2003}
M.~Dhiflaoui, S.~Funke, C.~Kwappik, K.~Mehlhorn, M.~Seel, E.~Sch\"{o}mer,
  R.~Schulte, and D.~Weber.
\newblock Certifying and repairing solutions to large {LP}s: How good are
  {LP}-solvers?
\newblock In \emph{Proceedings of the 14th annual ACM-SIAM symposium on
  Discrete algorithms}, SODA '03, pages 255--256. SIAM, 2003.

\bibitem[Edmonds(1967)]{Edmonds1967}
J.~Edmonds.
\newblock Systems of distinct representatives and linear algebra.
\newblock \emph{Journal of Research of the National Bureau of Standards},
  71B\penalty0 (4):\penalty0 241--245, 1967.

\bibitem[Edmonds and Maurras(1997)]{EdmondsMaurras1997}
J.~Edmonds and J.-F. Maurras.
\newblock {Note sur les $Q$-matrices d'Edmonds.}
\newblock \emph{{RAIRO. Recherche Op\'erationnelle}}, 31\penalty0 (2):\penalty0
  203--209, 1997.
\newblock URL \url{http://www.numdam.org/item?id=RO_1997__31_2_203_0}.

\bibitem[Escobedo and Moreno-Centeno(2015)]{EscobedoMorenoCenteno2015}
A.~R. Escobedo and E.~Moreno-Centeno.
\newblock Roundoff-error-free algorithms for solving linear systems via
  {C}holesky and {LU} factorizations.
\newblock \emph{INFORMS Journal on Computing}, 27\penalty0 (4):\penalty0
  677--689, 2015.
\newblock \doi{10.1287/ijoc.2015.0653}.

\bibitem[Escobedo and Moreno-Centeno(2017)]{EscobedoMorenoCenteno2017}
A.~R. Escobedo and E.~Moreno-Centeno.
\newblock Roundoff-error-free basis updates of {LU} factorizations for the
  efficient validation of optimality certificates.
\newblock \emph{SIAM Journal on Matrix Analysis and Applications}, 38\penalty0
  (3):\penalty0 829--853, 2017.
\newblock \doi{10.1137/16M1089630}.

\bibitem[Espinoza(2006)]{Espinoza2006}
D.~G. Espinoza.
\newblock \emph{On Linear Programming, Integer Programming and Cutting Planes}.
\newblock PhD thesis, Georgia Institute of Technology, 2006.
\newblock URL \url{http://hdl.handle.net/1853/10482}.

\bibitem[G\"artner(1999)]{Gaertner1999}
B.~G\"artner.
\newblock Exact arithmetic at low cost -- a case study in linear programming.
\newblock \emph{Computational Geometry}, 13\penalty0 (2):\penalty0 121--139,
  1999.
\newblock \doi{10.1016/S0925-7721(99)00012-7}.

\bibitem[Gleixner et~al.(2016)Gleixner, Steffy, and
  Wolter]{GleixnerSteffyWolter2016}
A.~M. Gleixner, D.~E. Steffy, and K.~Wolter.
\newblock Iterative refinement for linear programming.
\newblock \emph{INFORMS Journal on Computing}, 28\penalty0 (3):\penalty0
  449--464, 2016.
\newblock \doi{10.1287/ijoc.2016.0692}.

\bibitem[Gr{\"o}tschel et~al.(1988)Gr{\"o}tschel, Lov{\'a}sz, and
  Schrijver]{GroetschelLovaszSchrijver1988}
M.~Gr{\"o}tschel, L.~Lov{\'a}sz, and A.~Schrijver.
\newblock \emph{Geometric Algorithms and Combinatorial Optimization}, volume~2
  of \emph{Algorithms and Combinatorics}.
\newblock Springer, Berlin / Heidelberg, 1988.

\bibitem[Karmarkar(1984)]{Karmarkar1984}
N.~Karmarkar.
\newblock A new polynomial-time algorithm for linear programming.
\newblock \emph{Combinatorica}, 4\penalty0 (4):\penalty0 373--395, 1984.
\newblock \doi{10.1007/BF02579150}.

\bibitem[Khachiyan(1980)]{Khachiyan1980}
L.~G. Khachiyan.
\newblock Polynomial algorithms in linear programming (in {R}ussian).
\newblock \emph{Zhurnal Vychislitel'noi Matematiki i Matematicheskoi Fiziki},
  20\penalty0 (1):\penalty0 51--68, 1980.
\newblock \doi{10.1016/0041-5553(80)90061-0}.
\newblock English translation: \emph{{USSR} Computational Mathematics and
  Mathematical Physics}, 20(1):53--72, 1980.

\bibitem[Koch(2004)]{Koch2004}
T.~Koch.
\newblock The final {NETLIB-LP} results.
\newblock \emph{Operations Research Letters}, 32\penalty0 (2):\penalty0
  138--142, 2004.
\newblock \doi{10.1016/S0167-6377(03)00094-4}.

\bibitem[Kwappik(1998)]{Kwappik1998}
C.~Kwappik.
\newblock Exact linear programming.
\newblock Master's thesis, Universit\"at des Saarlandes, May 1998.

\bibitem[Lenstra et~al.(1982)Lenstra, Lenstra, and
  Lov\'asz]{LenstraLenstraLovasz1982}
A.~K. Lenstra, H.~W. Lenstra, and L.~Lov\'asz.
\newblock Factoring polynomials with rational coefficients.
\newblock \emph{Mathematische Annalen}, 261\penalty0 (4):\penalty0 515--534,
  1982.
\newblock \doi{10.1007/BF01457454}.

\bibitem[M\'esz\'aros()]{MeszarosLp}
C.~M\'esz\'aros.
\newblock {LP} test set.
\newblock URL \url{http://www.sztaki.hu/~meszaros/public_ftp/lptestset/}.

\bibitem[Mittelmann()]{MittelmannLp}
H.~Mittelmann.
\newblock {LP} test set.
\newblock URL \url{http://plato.asu.edu/ftp/lptestset/}.

\bibitem[Nguyen and Stehl\'e(2009)]{NguyenStehle2009}
P.~Q. Nguyen and D.~Stehl\'e.
\newblock An {LLL} algorithm with quadratic complexity.
\newblock \emph{SIAM Journal on Computing}, 39\penalty0 (3):\penalty0 874--903,
  2009.
\newblock \doi{10.1137/070705702}.

\bibitem[Renegar(1988)]{Renegar1988}
J.~Renegar.
\newblock A polynomial-time algorithm based on {N}ewton's method, for linear
  programming.
\newblock \emph{Mathematical Programming}, 40\penalty0 (1--3):\penalty0 59--93,
  1988.
\newblock \doi{10.1007/BF01580724}.

\bibitem[Schrijver(1986)]{Schrijver1986}
A.~Schrijver.
\newblock \emph{Theory of Linear and Integer Programming}.
\newblock John Wiley \& Sons, 1986.

\bibitem[Steffy(2011)]{Steffy2011}
D.~E. Steffy.
\newblock Exact solutions to linear systems of equations using output sensitive
  lifting.
\newblock \emph{ACM Communications in Computer Algebra}, 44\penalty0
  (3/4):\penalty0 160--182, 2011.
\newblock \doi{10.1145/1940475.1940513}.

\bibitem[{University of Tennessee Knoxville and Oak Ridge National
  Laboratory}()]{netlib}
{University of Tennessee Knoxville and Oak Ridge National Laboratory}.
\newblock {Netlib LP Library}.
\newblock URL \url{http://www.netlib.org/lp/}.

\bibitem[Wang and Pan(2003)]{WangPan2003}
X.~Wang and V.~Y. Pan.
\newblock Acceleration of euclidean algorithm and rational number
  reconstruction.
\newblock \emph{SIAM Journal on Computing}, 2\penalty0 (32):\penalty0 548--556,
  2003.

\bibitem[{Zuse Institute Berlin}({\natexlab{a}})]{MiplibWeb}
{Zuse Institute Berlin}.
\newblock {MIPLIB---Mixed Integer Problem Library}, {\natexlab{a}}.
\newblock URL \url{http://miplib.zib.de/}.

\bibitem[{Zuse Institute Berlin}({\natexlab{b}})]{SoPlexWeb}
{Zuse Institute Berlin}.
\newblock {SoPlex. Sequential object-oriented simPlex}, {\natexlab{b}}.
\newblock URL \url{http://soplex.zib.de/}.

\end{thebibliography}

\end{small}

\end{document}